\documentclass[11pt,reqno]{amsart}
\usepackage{amsmath, amssymb, amsthm}
\usepackage{url}
\usepackage{mathrsfs}
\usepackage[ansinew]{inputenc}
\usepackage[breaklinks]{hyperref}
\usepackage{fancyhdr}

\usepackage{caption}

\usepackage{multirow}

\allowdisplaybreaks

\newcommand{\pFq}[5]{\ensuremath{{}_{#1}{}\phi_{#2} \left( \genfrac{}{}{0pt}{}{#3}
{#4} ; {#5} \right)}}

\renewcommand{\(}{\left\(}
\renewcommand{\)}{\right\)}
\renewcommand{\[}{\left\[}
\renewcommand{\]}{\right\]}
\newtheorem{remark}[]{Remark}
\numberwithin{equation}{section}
 \theoremstyle{plain}
\newtheorem{theorem}{Theorem}[section]
\newtheorem{lemma}[theorem]{Lemma}
\newcommand{\sm}{\left(\begin{smallmatrix}}
\newcommand{\esm}{\end{smallmatrix}\right)}

\newtheorem{corollary}[theorem]{Corollary}
\newtheorem{proposition}[theorem]{Proposition}

   \makeatletter
\def\proof{\@ifnextchar[{\@oproof}{\@nproof}}
\def\@oproof[#1][#2]{\trivlist\item[\hskip\labelsep\textit{#2 \textbf{Proof of}\
#1.}~]\ignorespaces}
\def\@nproof{\trivlist\item[\hskip\labelsep\textit{Proof.}~]\ignorespaces}

\makeatother

\usepackage{color}
\usepackage{amsmath}

\definecolor{blue}{rgb}{0,0,1}
\definecolor{red}{rgb}{1,0,0}
\definecolor{green}{rgb}{0,.6,.2}
\definecolor{purple}{rgb}{1,0,1}

\long\def\red#1\endred{{\color{red}#1}}
\long\def\blue#1\endblue{{\color{blue}#1}}
\long\def\purple#1\endpurple{{\color{purple}#1}}
\long\def\green#1\endgreen{{\color{green}#1}}

\begin{document}
\title[partitions with parts separated by parity]{
	 Untrodden pathways in the study of partitions with parts separated by parity}

\author{Nargish Punia}
\address{Department of Mathematics, Indian Institute of Technology, Roorkee-247667, Uttarakhand, India}
\email{nargish@ma.iitr.ac.in}

\author{Rahul Kumar}
\address{Department of Mathematics, Indian Institute of Technology, Roorkee-247667, Uttarakhand, India}
\email{rahul.kumar@ma.iitr.ac.in}

\maketitle
\pagenumbering{arabic}
\pagestyle{headings}

\begin{abstract}
In this paper, we study a partition function with parts separated by parity from a family of partition functions introduced by Andrews in 2019. We first show that its generating function admit representation in terms of mock modular forms of weight $3/2$. We then initiate a study of smallest-part partition function analogues of partitions with parts separated by parity, which, to the best of our knowledge, has not previously been investigated in the literature. We show that the generating functions of these smallest-part partition functions are likewise expressible in terms of mock modular forms, mock Maass theta functions and Zagier's $q$-series $H(q)$, which he describes as a ``half-derivative of the Dedekind eta-function".  Our study also leads to new combinatorial interpretations of the coefficients of the aforementioned mock modular form,  function $\sigma^*(q)$, and mock theta functions $\psi(q)$ and $A(q)$ in terms of odd Ferrers graphs. Moreover, we provide a quantum $q$-series identity involving Ramanujan's famous $\sigma(q)$ function. We also derive the leading asymptotic term of one of these partition functions, whose generating function we first express in terms of $\sigma^*(q)$.

\end{abstract}
\tableofcontents

\section{Introduction}\label{intro}

A partition is said to have its parts \emph{separated by parity} if either all of its even parts are larger than all of its odd parts, or vice versa. The study of such integer partitions began with Andrews' influential paper \cite{andrews parity1}. In that paper, he introduced the partition function $\mathcal{EO}(n)$, which counts the number of partitions of $n$ in which every even part is smaller than each odd part. Moreover, let $\overline{\mathcal{EO}}(n)$ denote the number of partitions counted by $\mathcal{EO}(n)$ in which only the largest even part appears an odd number of times (we allow partitions with no
even parts where we tacitly assume 0 is the largest even part and appears once). Then Andrews proved that the generating function of this partition function is related with the third order mock theta function of Ramanujan \cite[p.~62]{watson}
\begin{align}\label{nuq}
\nu(q):=\sum_{n=0}^\infty\frac{q^{n^2+n}}{(-q;q^2)_{n+1}},
\end{align}
through the following relation \cite[Theorem 1.1]{andrews parity1}:
\begin{align}
\nu(q)+\nu(-q)=2\overline{\mathrm{eo}}(q),\nonumber
\end{align}
where \cite[Section 3]{andrews parity1}
\begin{align*}
\overline{\mathrm{eo}}(q):=\sum_{n=0}^\infty\overline{\mathcal{EO}}(n)q^n=\sum_{n=0}^{\infty}\frac{q^{2n}}{(q^4;q^4)_n(q^{4n+2};q^4)_\infty}.
\end{align*}
Here, and throughout the sequel, we use the following standard notations:
\begin{align*}
	(A;q)_0&=1,\\
	(A;q)_N&=(1-A)(1-Aq)\cdots(1-Aq^{N-1}),\qquad 1\leq N<\infty,\\
	(A;q)_\infty&=\lim_{N\to\infty}(A;q)_N,\quad |q|<1.
\end{align*}

One may then consider generating functions that enumerate such partitions with  subject to additional restrictions. For instance, using the notation introduced by Andrews \cite{andrews parity2}, $p_{ed}^{ou}(n)$ denotes the number of partitions of $n$ in which every odd part exceeds each even part, the even parts are distinct, and the odd parts are unrestricted. Numerous partition functions of this type can be defined by imposing different parity conditions, and their generating functions have been studied extensively (\cite{andrews parity2, bcn, bjs, Fu-Tang} to name a few). All the eight partition functions arising from Andrews' notation exhibit remarkable connections with modular forms, mock modular objects, and false-indefinite theta functions.  


Despite the rich theory surrounding $\overline{\mathcal{EO}}(n)$, comparatively little is known about its analogues corresponding to the other aforementioned partition classes. Apart from the work of Andrews \cite{andrews parity1} and  Fu and Tang \cite{Fu-Tang}, such partition functions have received very little attention. 

The main objective of this paper is twofold. First, we undertake a study of the analogue of Andrews' partition function $\overline{\mathcal{EO}}(n)$ for its mirror partition class and show that it possesses a remarkably rich theory and  leading to several surprising results.  Among other results, we prove that its generating function can be expressed in terms of a mock modular form and provide a natural combinatorial interpretation of the coefficients of this mock modular form via odd  Ferrers graphs of Andrews. Further, this study yields  several new representations of $\sigma^*(q)$, the companion of Ramanujan's celebrated function $\sigma(q)$. These representations lead to a new combinatorial interpretation of the coefficients of $\sigma^*(q)$, as well as a quantum $q$-series identity for $\sigma(q)$.

The second objective of this paper is to initiate a study of smallest-part analogues of partitions with parts separated by parity which has not previously been investigated in the literature. We show that the generating functions of these smallest-part partition functions are likewise expressible in terms of mock modular forms, mock Maass theta functions. See Section \ref{spt functions} below.

We begin with the first part of our investigation by introducing the necessary notations. We consider the Andrews' partition function\footnote{Note the difference in the order of $\mathcal{E}$ and $\mathcal{O}$ in the notation.} $\mathcal{OE}(n)$, which counts the partitions of $n$ in which every odd part is smaller than each even part. Then, we have \cite{andrews parity2}
\begin{align*}
\mathrm{oe}(q):=\sum_{n=1}^\infty \mathcal{OE}(n)q^n&=\sum_{n=1}^\infty\frac{q^{2n-1}}{(q;q^2)_{n}(q^{2n};q^2)_\infty}=\frac{1}{1-q}\left(\frac{1}{(q;q^2)_\infty}-\frac{1}{(q^2;q^2)_\infty}\right).
\end{align*}
Now, let $\overline{\mathcal{OE}}(n)$ denote the number of partitions counted by $\mathcal{OE}(n)$ in which only largest odd part appear odd number of times. Its generating function can be given as
\begin{align}\label{oe bar generating}
\overline{\mathrm{oe}}(q):=\sum_{n=1}^\infty \overline{\mathcal{OE}}(n)q^n=\sum_{n=1}^{\infty}\frac{q^{2n-1}}{(q^2;q^4)_n(q^{4n};q^4)_\infty}.
\end{align}

Among the many surprises of this study is our first result which shows that the generating function $\overline{oe}(q)$ can be expressed in terms of a mock modular form of weight $3/2$. Recall that the holomorphic part of a harmonic Maass form is known as a \emph{mock modular form}. Note that  a harmonic Maass form is a function that satisfies modularity, is annihilated by the weight $k$ hyperbolic Laplacian,	and satisfies a growth condition at the cusps (see \cite{bfor} for more details). 

\begin{theorem}\label{oebar thm}
The $q$-series $\overline{oe}(q)$ is a mixed mock modular form\footnote{Recall that a mixed mock modular form is a linear combination of products of modular forms and mock modular forms.}. More precisely, we have,
\begin{align}
\overline{\mathrm{oe}}(q)=\frac{1}{q(q^4;q^4)_\infty}\mathcal{U}(q^2).\nonumber
\end{align}
where,
\begin{align}\label{u defn}
\mathcal{U}(q)=\sum_{n=1}^{\infty}\frac{(-1)^{n+1}q^{n^2}}{(1-q^{2n-1})(q;q^2)_{n}}.
\end{align} 
\end{theorem}
We here note that the function $\mathcal{U}(q)$ is a mock modular form of weight $3/2$ with shadow $\frac{i}{2\pi}\frac{\eta^2(2\tau)}{\eta(\tau)}$, as shown by Alfes, Bringmann and Lovejoy \cite[Thoerem 4.5]{abl} (also see \cite[Section 6]{blo}). More precisely, they studied the modular properties of the $q$-series: 
\begin{align}\label{N}
N^o(a,b;z,q):=\sum_{n=0}^\infty\frac{(-q/a;q^2)_n(-q/b;q^2)_n(ab)^nq^{2n+1}}{(zq;q^2)_{n+1}(q/z;q^2)_{n+1}}.
\end{align}
Clearly (letting $a\to0$, $b=-1$ and $z=1$ in \eqref{N}),
\begin{align}\label{N in terms of U1}
N^o(0,-1;1,q)=\mathcal{U}(q).
\end{align}
The function $N^o(0,-1;1,q)$ or $\mathcal{U}(q)$ also appears in the work of Chen and Chen \cite{cc} where they find Heck-Rogers type representation for it. 

Combinatorial  interpretation of $\mathcal{U}(q)$ is the subject of our next result. To that end, we recall that the odd Ferrers graphs, introduced by Andrews \cite[Section 8, pp.~60--62]{andrews invent} (also see \cite[Section 4]{andrews parity1}), which consist of a Ferrers graph using 2’s with a surrounding border of 1’s. For example,  the partition $6+7+7+5+1$ can be represented through Andrews' odd Ferrers graph as:
\begin{align*}
\begin{array}{r *{8}{c}} 
	1& 	1 & 1& 1 & 1 & 1 &  & \\
	\cline{2-3}
	1& 	\multicolumn{1}{|c}{2} & \multicolumn{1}{c|}{2} & 2  &   &   &   &   \\
	1& \multicolumn{1}{|c}{2} & \multicolumn{1}{c|}{2} & 2  &  &   &   &   \\
	\cline{2-3} 
	1 & 2 & 2 &  &   &   &   &   \\
	1 &  &   &   &   &   &   &   
\end{array}
\end{align*}
As Andrews \cite[p.~60]{andrews invent} noted that generating function of odd Ferrers graph is $q \omega(q)$, where $\omega(q)$ is the Ramanujan's third order mock theta function 
\begin{align*}
\omega(q)=\sum_{n=1}^\infty\frac{q^{n-1}}{(q;q^2)_{n}}.
\end{align*}

By Durfee square, we here mean by a largest square which is under the boundary made up of 1's. Moreover,  self-conjugate odd Ferrers graphs are ones that are unaltered under conjugation (The conjugate of an odd Ferrers graph is obtained by interchanging rows and columns). As shown in \cite[p.~68]{andrews invent} (see \cite[pp.439-440]{andrews parity1}), the generating function of self-conjugate odd Ferrers graphs is $q\nu(-q^2)$, where $\nu(q)$ is defined in \eqref{nuq}.

We next introduce certain variations of odd Ferrers graphs that are not merely artificial constructions but help us to explain the combinatorial interpretations of two mock theta functions of orders 2 and 3, the mock modular form $\mathcal{U}(q)$, and the mock Maass theta function $\sigma^*(q)$. We begin with the first of these variations by defining   an odd Ferrers graph to be a \emph{Durfee odd Ferrers graph} if the entries along the boundary of the Durfee square are 2's, while the remaining entries on the boundary of the odd Ferrers graph are 1's. For example,  the partition $11+8+8+5+3$ can be represented by the following Durfee odd Ferrers graph:
\begin{align*}
\begin{array}{r *{8}{c}} 
	2& 	2 & 2& 1 & 1 & 1 & 1 & 1 \\
	\cline{2-3}
	2& 	\multicolumn{1}{|c}{2} & \multicolumn{1}{c|}{2} & 2  &   &   &   &   \\
	2& \multicolumn{1}{|c}{2} & \multicolumn{1}{c|}{2} & 2  &  &   &   &   \\
	\cline{2-3} 
	1 & 2 & 2 &  &   &   &   &   \\
	1 & 2 &  &   &   &   &   &  
\end{array}
\end{align*}
It can observed that the generating function for Durfee odd Ferrers graphs is given by $\sum_{n=1}^\infty q^{2n^2}/(q;q^2)_n^2$, whereas the number of self-conjzugate Durfee odd Ferrers graph are coefficients of Ramanujan's another third mock theta function \cite[p.~5, Equation ]{andrews rlnb}, namely,
\begin{align*}
	\psi(q^2)=\sum_{n=1}^\infty\frac{q^{2n^2}}{(q^2;q^4)_n}.
\end{align*}

We now define \emph{colored Durfee odd Ferrers graph}, it is a Durfee odd Ferrers graph in which only the largest row of the form $(2n-1)$ below the Durfee square of size $(n-1)$ and the largest column of size $(2n-1)$ to its right may appear in two colors.

We first provide a combinatorial interpretation for $\mathcal{U}(q^2)$ in terms of colored Durfee odd Ferrers graph and then one for $\mathcal{U}(q)$.
\begin{theorem}\label{combinatorics of U2}
 $\mathcal{U}(q^2)$ is a generating function for excess number of self conjugate colored Durfee odd Ferrers graphs with even size Durfee square over with odd size Durfee square.
	
	
\end{theorem}

The next corollary shows that the excess number of partitions in the above theorem is given by Hurwitz class numbers corresponding to a certain arithmetic progression.
\begin{corollary}\label{hurwitz combinatorics}
	Let $u_{e}(n)$ \textup{(}resp. $u_{o}(n)$\textup{)} denote the number of self conjugate colored Durfee odd Ferrers graphs with even \textup{(}resp. odd\textup{)} size Durfee square. Then, for \emph{even} $n$, we have
	\begin{align}\label{hurwitz combinatorics eqn}
		u_{e}(n)-u_{o}(n)=H(4n-1),
	\end{align}
where $H(n)$ are Hurwitz class numbers  of binary quadratic forms of discriminant $-n$. 
\end{corollary}

While the preceding results concern $\mathcal{U}(q^2)$, the next result is devoted to $\mathcal{U}(q)$. To that end, we first derive a new $q$-series representation for $\mathcal{U}(q)$, which subsequently allows us to obtain its combinatorial interpretation. For that, let us introduce a refined version of Andrews' odd Ferrers graph, which we call a  \emph{Refined odd Ferrers graph}. These are odd Ferrers graphs in which, except for the first entry, the first column consists of both 0's and 1's, and no two rows containing a 0 are identical\footnote{Obviously, a row consisting of just the single entry $0$ is not allowed.}. For example,
\begin{align*}
	\begin{array}{r *{8}{c}} 
		1& 	1 & 1& 1 & 1 & 1 & 1 & 1 \\
		0& 	2 & 2 & 2  &   &   &   &   \\
		1& 2 & 2 & 2  &  &   &   &   \\
		1 & 2 & 2 &  2&   &   &   &   \\
		0& 2 &  &   &   &   &   &   
	\end{array}
\end{align*}
Interestingly, it can be seen that the coefficients of McIntosh's second-order mock theta function \cite[Section 2]{mc}:
\begin{align*}
A(q):=\sum_{n=0}^\infty \frac{(-q^2;q^2)_n}{(q;q^2)_{n+1}}q^{n+1}.
\end{align*}
enumerate Refined odd Ferrers graphs.

Let $v_{e}(n)$ \textup{(}resp. $v_{o}(n)$\textup{)} denote the number of Refined odd Ferrers graphs having an even \textup{(}resp. odd\textup{)} number of rows with first entry $0$.
\begin{theorem}\label{U in terms of difference of odd Ferrers}
	We have,
		\begin{align}\label{new representation of U}
		\mathcal{U}(q)=\sum_{n=0}^{\infty}\frac{(q^2;q^2)_n}{(q;q^2)_{n+1}}q^{n+1}.
	\end{align}
Consequently, $\mathcal{U}(q)$ is the generating function for the difference $v_{e}(n)-v_{o}(n)$. Moreover, for $n\geq1$, we have
\begin{align}\label{U in terms of difference of odd Ferrers eqn}
v_{e}(n)-v_{o}(n)=H(8n-1).
\end{align}
\end{theorem}
There is another interpretation of $\ \mathcal{U}(q)$ which we provide next. Let $w(n)$ denote the number of partitions of $n$ into consecutive odd parts (with multiplicity) where the largest part may appear in two colors. Moreover, let $w_{e}(n)$ (resp. $w_{o}(n)$)  denote the number of partitions of $n$ counted by $w(n)$ such that total number of odd parts (without multiplicity) are even (resp. odd) in counting.

\begin{theorem}\label{U in terms of difference of simple}
	The function $\mathcal{U}(q)$ is the generating function for the difference $w_{o}(n)-w_{e}(n)$. 
	
\end{theorem}



Regarding interpretations of Hurwitz class numbers in terms of partition classes arising in the form given in \eqref{hurwitz combinatorics eqn} and \eqref{U in terms of difference of odd Ferrers eqn}, results of this type have been obtained previously in the literature. For example, Bringmann and Lovejoy \cite[Corollary 1]{bl} established such interpretations in terms of overpartitions and the M2-rank, while Chen and Garvan \cite{cg} obtained analogous interpretations for the coefficients of second-order mock theta functions. These results, together with our Corollary \ref{hurwitz combinatorics}, yield further results, which we state in Corollaries \ref{partition difference1} and \ref{partition difference2} below. 

We state the following partition identities, which follow by combining the above results for $\mathcal{U}(q)$, and which appear rather surprising given the disparate definitions of the corresponding partition functions.
\begin{corollary}\label{U combi}
We have,
\begin{align}\label{1t}
v_e(n)-v_o(n)=w_o(n)-w_e(n)=H(8n-1)=Q(n),
\end{align}
where $Q(n)$ counts the number of partitions of $n$ into consecutive parts, all singletons except possibly the largest \textup{(}see $A321440$ on OEIS or \cite[Equation (4.3)]{cc}\textup{)}.
\end{corollary}

We now observe that the definition of $\mathcal{U}(q)$ resembles with that of the famous function $\sigma^*(q)$ which is defined as\footnote{Usually, there is an extra factor of $2$ in the definition of $\sigma^*(q)$.} \cite{cohen}
\begin{align}\label{sigma*}
\sigma^*(q)=\sum_{n=1}^{\infty}\frac{(-1)^nq^{n^2}}{(q;q^2)_{n}}.
\end{align}
It is well-known that Cohen \cite{cohen} constructed Maass wave forms using $\sigma^*(q)$ and its companion function $\sigma(q)$ of Ramanujan \cite{ram}:
\begin{align}\label{sigma}
\sigma(q)=\sum_{n=0}^{\infty}\frac{q^{\frac{1}{2}n(n+1)}}{(-q;q)_n}.
\end{align}

Therefore, the combinatorial interpretation of $\mathcal{U}(q^2)$ established in Theorem~\ref{combinatorics of U2} naturally raises the question of whether $\sigma^*(q^2)$ admits a similar interpretation. The following theorem answers this question affirmatively by providing such an interpretation, which, to the best of our knowledge, is new.
\begin{theorem}\label{combinatorics of sigma*}
$\sigma^*(q^2)$ is the generating function for the excess number of self-conjugate Durfee odd Ferrers graphs with an odd size Durfee square over those with an even size Durfee square.
\end{theorem}

As we now have a combinatorial interpretation of $\sigma^*(q^2)$ in terms of Ferrers graphs, it is natural to seek one for $\sigma^*(q)$. For this purpose, we use the following representation derived by Cohen \cite[Equation (2)]{cohen}:
\begin{align}\label{sigma* repres}
	\sigma^*(q)=-\sum_{n=0}^{\infty}(q^2;q^2)_nq^{n+1}.
\end{align}
\begin{theorem}\label{combinatorics of sigma* in refined}
	$\sigma^*(q)$ is the generating function for the excess number of Refined  odd Ferrers graphs, where all entries \textup{(}except first\textup{)} of the first column are zero and odd  in counting over even in counting.
\end{theorem}

Recall that $\sigma(q)$ and $\sigma^*(q)$ are examples of \emph{mock Maass theta functions}, a term coined by Zwegers \cite{Zwegers}. More generally, Zwegers introduced a class of functions associated with indefinite binary quadratic forms that are ``nearly'' Maass waveforms and called them mock Maass theta functions. Moreover, \emph{mixed mock Maass theta functions} are linear combinations of products of modular forms and mock Maass theta functions (see \cite{bbd}).


In what follows, we provide a new combinatorial example of a mixed mock Maass theta function. To this end, let $k(n)$ denote the number of partitions of $n$ in which the largest part is an odd integer greater than $1$, all other parts are even, and the largest part is not consecutive with any other part. Note that the generating function of $k(n)$ can be given by
\begin{align}\label{kn}
\mathcal{K}(q):=\sum_{n=1}^\infty k(n)q^n=\sum_{n=0}^{\infty}\frac{q^{2n +3}}{(1-q^{2n +3})(q^2;q^2)_n}.
\end{align}

Our next result shows that $\mathcal{K}(q)$ can be expressed in terms of $\sigma^*(q)$, thereby showing that $\mathcal{K}(q)$ is an example of a mixed mock Maass theta function.
\begin{theorem}\label{sigma* as generating function}
We have,
\begin{align}
\mathcal{K}(q)=\frac{1-q}{q(q^2;q^2)_\infty}\sigma^*(q)+\frac{1}{(q^2;q^2)_\infty}.\nonumber
\end{align}
\end{theorem}

The above results about $\sigma^*(q)$ and $\sigma^*(q^2)$ are not the only outcomes of our study. Our further search also yields the following two new representations of $\sigma^*(q)$, which, to the best of our knowledge, have not explicitly appeared in the literature. The significance of these representations lies in the fact that one provides a new combinatorial interpretation of the coefficients of $\sigma^*(q)$ , while the other yields a new quantum $q$-series identity for $\sigma(q)$.
\begin{theorem}\label{sigma* representations}
We have,
	\begin{align}
	&(1)&\sigma^{*}(q)&=-\sum_{n=0}^{\infty}\frac{q^{2n+1}}{(1-q^{2n+1})}(q^{2n+2};q^2)_\infty. \label{representation 1} \\
	&(2)&\sigma^{*}(q)&=-\frac{q}{1-q}\left(1-\sum_{n=1}^{\infty}(q^2;q^2)_{n-1}q^{3n}\right). \label{representation 2}
	\end{align}
\end{theorem}

Identity \eqref{representation 1} in the above theorem immediately yields the following corollary.
\begin{corollary}\label{sigma* combin}
The function $\sigma^{*}(q)$ is the generating function for excess number of partitions having smallest part odd and other parts even and distinct with odd in counting over even in counting.
\end{corollary}

Before stating the aforementioned quantum $q$-series identity, we note that by a quantum $q$-series identity, we mean an identity that does not hold as an equality of power series inside the unit disk in the classical sense, but does hold when $q$ is a root of unity. Prototypical examples were given by Cohen \cite{cohen}:
\begin{align}\label{quantum1}
\sigma(q)=-\sigma^*\left(q^{-1}\right),
\end{align}
and\footnote{These identities originally contain an extra factor of $2$, which does not appear in \eqref{quantum2} due to our definition of $\sigma^*(q)$ in \eqref{sigma*}.}
\begin{align}\label{quantum2}
\sigma^*\left(q^{-1}\right)=\sum_{n=0}^\infty(q;q)_n,\qquad \sigma^*(q)=-\sum_{n=0}^\infty(q;q^2)_n,
\end{align}
where \eqref{quantum1} is valid when $q$ is a root of unity and  the first identity of \eqref{quantum2}  holds when  $q$ a primitive even root of unity, while the second holds when $q$ a primitive odd root of unity. Note that the above identities are not valid in the unit disk in the classical sense, since none of the right-hand sides defines a convergent power series there. Such quantum $q$-series identities have deep connections to quantum modular forms and quantum topology \cite{bopr, hj, Zagier quantum}. After the appearance of the identities \eqref{quantum1} and \eqref{quantum2}, numerous results of a similar nature have been established by various authors, including Bryson, Ono, Pitman, and Rhoades \cite{bopr}, Folsom, Ki, Vu, and Yang \cite{fk}, Lovejoy \cite{lovejoy} to name a few. Our next result provide such a quantum $q$-series identity for $\sigma(q)$ as an application of \eqref{representation 2}.
\begin{theorem}\label{quantum series}
If $q$ is a primitive odd root of unity, then, we have 
\begin{align}\label{quantum series eqn}
\sigma\left(q^{-1}\right)=\frac{q}{1-q}\left(1-\sum_{n=1}^{\infty}(q^2;q^2)_{n-1}q^{3n}\right).
\end{align}
\end{theorem}

We end this subsection by deriving the main asymptotic terms for partition function $k(n)$. A classical example of such an asymptotic formula  is the celebrated result of Hardy and Ramanujan \cite{hr} for the unrestricted partition function $p(n)$, which counts the number of partitions of $n$. By developing the Circle method, they obtained an asymptotic expansion for $p(n)$, whose leading term is given by
\begin{align*}
p(n)\sim\frac{1}{4n\sqrt{3}}e^{\pi\sqrt{\frac{2n}{3}}}\qquad \mathrm{as}\quad n\to\infty.
\end{align*}


The main asymptotic term of the partition function $k(n)$, defined in \eqref{kn}, is given in the next result.
\begin{theorem}\label{asymptotic of kn}
As $n\to\infty$, we have
	\begin{align*}
k(2n-1)\sim \frac{1}{4n\sqrt{3}} e^{\pi \sqrt{\frac{2n}{3}}} \qquad\mathrm{and}\qquad	k(2n)\sim \frac{\pi}{12\sqrt{2}n^{\frac{3}{2}}}e^{\pi\sqrt{\frac{2n}{3}}}.
\end{align*}
\end{theorem}

\subsection{Smallest Part analogues for Partitions with parts separated by parity}\label{spt functions}
As we discussed above, partitions associated with the parity of parts have been studied extensively in the literature. However, their smallest part partition function analogues have not been previously considered in the literature. In this subsection, we initiate a systematic study of these analogues and show that their generating functions are mixed mock modular forms. 

 Recall that smallest part partition function $\operatorname{spt}(n)$, introduced by Andrews \cite{andrews spt}, counts the total number of occurrences of the smallest parts among all the partitions of $n$. Note that
 \begin{align*}
 	&\sum_{n=1}^\infty\mathrm{spt}(n)q^n=\sum_{n=1}^\infty\frac{q^{n}}{(1-q^{n})^2(q^{n+1};q)_\infty}.
 \end{align*}
 Since its appearance, spt$(n)$ has emerged as one of the central objects in the theory of partitions exhibiting remarkable arithmetic and asymptotic properties analogous to those of the unrestricted partition function $p(n)$ (see, for example, \cite{aa, bringmann duke, dy, fo, o}, to name a few). While the generating function of $p(n)$ is essentially modular, Bringmann \cite{bringmann duke} proved that the generating function of $\operatorname{spt}(n)$ is the holomorphic part of a weight $3/2$ harmonic Maass form, or equivalently, a mock modular form whose shadow is the Dedekind eta-function. 
 
 
 Motivated by these interesting connections, we show that the generating functions of the parity-restricted smallest-part functions introduced in this subsection also admit natural descriptions in terms of mock modular forms. 
 
 Our first such spt-function is the one corresponding to Andrews' function $\mathcal{EO}(n)$, introduced above. More precisely, let  $\mathrm{spt}_{\mathcal{EO}}(n)$ denote the number of smallest odd parts in all the partitions counted by\footnote{Note that the partitions of $n$ counted by $\mathcal{EO}(n)$ that contain no odd parts do not contribute to $\mathrm{spt}_{\mathcal{EO}}(n)$.}
  $\mathcal{EO}(n)$. Its generating function can be given by
 \begin{align}\label{sptEO}
\mathcal{S}_{\mathcal{EO}}(q):=\sum_{n=1}^\infty\mathrm{spt}_{\mathcal{EO}}(n)q^n=\sum_{n=0}^{\infty}\frac{q^{2n+1}}{(1-q^{2n+1})(q^2;q^2)_n(q^{2n+1};q^2)_\infty}.
 \end{align}
 
Quite surprisingly, $\mathcal{S}_{\mathcal{EO}}(q)$ can be expressed in terms of $\mathcal{U}(q)$, thereby showing that it is a mix mock modular form. More precisely,
\begin{theorem}\label{sptEO thm}
We have,
\begin{align}
 \mathcal{S}_{\mathcal{EO}}(q)=\frac{1}{(q^2;q^2)_\infty}\mathcal{U}(q).\nonumber
\end{align}
Consequently, $(q^2;q^2)_\infty\mathcal{S}_{\mathcal{EO}}(q)$ is a mock modular form of weight $3/2$.
\end{theorem}


Note Theorems \ref{oebar thm} and \ref{sptEO thm} together imply that
\begin{align}\label{eoeobar}
\mathcal{S}_{\mathcal{EO}}(q^2)=q\ \overline{oe}(q).
\end{align} 
Combinatorial interpretation of the identity \eqref{eoeobar} leads to the following interesting consequence.
\begin{corollary}\label{oesoe}
	Let $\overline{\mathcal{EO}}_1(n)$ denote the number of partitions of $n$ counted by $\mathcal{EO}(n)$ where all parts appear even number of times and $\mathcal{S}_{\overline{\mathcal{EO}}_1}(n)$ counts the smallest odd part present in all the partitions of $n$ counted by $\overline{\mathcal{EO}}_1(n)$. Then, for $n\geq1$, we have
\begin{align}
\overline{\mathcal{OE}}(n)=\mathcal{S}_{\overline{\mathcal{EO}}_1}(n+1).\nonumber
\end{align}
\end{corollary}

We now consider spt-analogue of the partition function $\mathcal{OE}(n)$. Namely, $\mathrm{spt}_{\mathcal{OE}}(n)$ denote the number of smallest even parts in all the partitions generated by $q \mathrm{oe}(q)$. Note that $q \mathrm{oe}(q)$ generates those partitions of $n$ where all odd parts are less than each even part and having atleast one even part. Clearly, 
\begin{align}\label{sptOE}
	\mathcal{S}_{\mathcal{OE}}(q):=\sum_{n=1}^\infty\mathrm{spt}_{\mathcal{OE}}(n)q^n=\sum_{n=1}^{\infty}\frac{q^{2n}}{(1-q^{2n})(q;q^2)_n(q^{2n};q^2)_\infty}.
\end{align}

Our next result shows that $\mathcal{S}_{\mathcal{OE}}(q)$ is also a mixed mock modular form.
\begin{theorem}\label{sptOE thm}
We have,
\begin{align}\label{sptOE thm eqn}
\mathcal{S}_{\mathcal{OE}}(q)
=\frac{1}{(q;q^2)_\infty}\mathcal{U}(q)-\frac{1}{(q^2;q^2)_\infty}\frac{E_2(2\tau)-E_2(\tau)}{24},
\end{align}
where $E_2(\tau)=1-24\sum_{n\geq1}\frac{q^n}{(1-q^n)^2},\ q=e^{2\pi i\tau},$ is the quasi-modular form.
\end{theorem}


Theorems \ref{sptEO thm} and \ref{sptOE thm} immediately yield the following corollary, which is reminiscent of relations among Ramanujan's mock theta functions. For example, one has the following relation between Ramanujan's mock theta functions \cite[Chapter 2, Entry 2.3.1]{rlnb-V}:
\begin{align}
	\phi_3(-q)-\frac{1}{2}f_3(q)=\frac{1}{2}\frac{(q;q)_\infty}{(-q;q)_\infty^2},\nonumber
\end{align}
where $\phi_3(q):=\sum_{n=0}^\infty\frac{q^{n^2}}{(-q^2;q^2)_n}$ and $f_3(q):=\sum_{n=0}^\infty\frac{q^{n^2}}{(-q;q)_n^2}$ are third order mock theta functions.

In a similar vein, our functions $\mathcal{S}_\mathcal{EO}(q)$ and $\mathcal{S}_\mathcal{OE}(q)$  satisfy the following relation:
\begin{corollary}\label{relation bw sptEO and sptOE}
	We have,
\begin{align*}
	\psi(q)\mathcal{S}_\mathcal{EO}(q)-\mathcal{S}_{\mathcal{OE}}(q)=\frac{E_2(2\tau)-E_2(\tau)}{24\ \eta(2\tau)}e^{\frac{\pi i\tau}{6}},
\end{align*}
where $\psi(q)=(q^2;q^2)_\infty/(q;q^2)_\infty$ is the Ramanujan's theta function and $\eta(\tau):=q^{1/24}(q;q)_\infty$ is the Dedekind eta function.
\end{corollary}
The final spt-analogue in this series corresponds to Andrews' function $\mathcal{EO}_{ud}(n)$ which counts the number of partitions of $n$ in which every even part is less than each odd part, and where odd parts must be distinct, while even parts are unrestricted  \cite{andrews parity2}. Let $\mathrm{spt}_{\mathcal{EO}_{ud}}(n)$ denote the number of largest even parts in all the partitions counted by\footnote{Note that the partitions of $n$ counted by $\mathcal{EO}_{ud}(n)$ that contain no even parts do not contribute to $\mathrm{spt}_{\mathcal{EO}_{ud}}(n)$.}
 $\mathcal{EO}_{ud}(n)$. Clearly,
\begin{align}\label{sptEOud}
\mathcal{S}_{\mathcal{EO}_{ud}}(q):=\sum_{n=1}^\infty\mathrm{spt}_{\mathcal{EO}_{ud}}(n)q^n=\sum_{n=1}^{\infty}\frac{q^{2n}(-q^{2n+1};q^2)_\infty}{(1-q^{2n})(q^2;q^2)_n}.
\end{align}

While we showed above that $\mathcal{S}_{\mathcal{EO}}(q)$ and $\mathcal{S}_{\mathcal{OE}}(q)$ are mixed mock modular forms, we now show that $\mathcal{S}_{\mathcal{EO}{ud}}(-q)$ is a mixed mock Maass theta function by expressing it in terms of Ramanujan's $\sigma(q)$ function and 	Zagier's $q$-series $H(q)$, which he describes as ``the half-derivative of the Dedekind eta function".
\begin{theorem}\label{sptEOud thm}
Let $\sigma(q)$ be defined in \eqref{sigma}. We have,
\begin{align}\label{sptEOud thm eqn}
	(-q;q)_\infty	\mathcal{S}_{\mathcal{EO}_{ud}}(-q)=D(q)-\frac{1}{4(q;q)_\infty}H(q)-\frac{1}{4(-q;q)_\infty}\sigma(q),
\end{align}
where $D(q)$ is defined as
\begin{align*}
D(q):=\frac{1}{2}-\sum_{n=1}^\infty d(n)q^n \qquad \big(d(n)=\sum_{d|n}1\big),
\end{align*}
and $H(q)$ is the $q$-series defined by Zagier \cite[p.~951, Equation (20)]{Zagier}:
\begin{align*}
H(q):=\sum_{n=1}^\infty n\chi(n)q^{(n^2-1)/24},
\end{align*}
with  $\chi(.)=(12/.)$ being the unique primitive character of conductor $12$.
\end{theorem}

To give a combinatorial interpretation of the left-hand side of \eqref{sptEOud thm eqn}, namely, the term $(-q;q)_\infty	\mathcal{S}_{\mathcal{EO}_{ud}}(-q)$, we recall a partition function introduced by Andrews and Newman \cite[Theorem 4]{an} in their study of partitions with minimal excludant. They studied the function $p_e(n)$, which counts the number of partitions of $n$ with an even number of parts, or equivalently, by conjugation, the number of partitions of $n$ whose largest part is even and other parts are unrestricted. Using the fact that $(-q;q)_\infty=1/(q;q^2)_\infty$, we have
\begin{align}\label{Le generating}
(-q;q)_\infty	\mathcal{S}_{\mathcal{EO}_{ud}}(-q)&=\sum_{n=1}^{\infty}\frac{q^{2n}(q^{2n+1};q^2)_\infty}{(q;q^2)_\infty(1-q^{2n})(q^2;q^2)_n}\nonumber\\
&=\sum_{n=1}^{\infty}\frac{q^{2n}}{(1-q^{2n})^2(q;q)_{2n-1}}\nonumber\\
&=\sum_{n=1}^\infty \mathcal{L}_{e}(n)q^n.
\end{align}
Thus,  $(-q;q)_\infty	\mathcal{S}_{\mathcal{EO}_{ud}}(-q)$ is the generating function for  $\mathcal{L}_{e}(n)$, where  $\mathcal{L}_{e}(n)$ denotes the total number of occurrences of the largest part among all the partitions counted by $p_e(n)$. 

We conclude this subsection by showing that $\mathcal{L}_{e}(n)$ is asymptotically equivalent to $p(n)$, the unrestricted partition function.
\begin{corollary}\label{Len asymptotic}
As $n\to\infty$, we have
\begin{align*}
\mathcal{L}_{e}(n)\sim\frac{1}{2}p(n).
\end{align*}
\end{corollary}

This paper is organized as follows. In Section \ref{section: preliminaries}, we collect several results from the literature that will be needed in the proofs of our main theorems. Section \ref{section: partitions with parity} is devoted to the proofs of Theorem \ref{oebar thm}, Corollary \ref{hurwitz combinatorics}, Theorem \ref{sigma* as generating function}, and Theorem \ref{sigma* representations}. Quantum $q$-series identity given in Theorem \ref{quantum series} is proved in Section \ref{section: quantum series}. In Section \ref{section: asymptotics}, we establish the asymptotic results stated in Theorem \ref{asymptotic of kn} and Corollary \ref{Len asymptotic}. The combinatorial results (Theorems \ref{combinatorics of U2}, \ref{U in terms of difference of odd Ferrers}, \ref{U in terms of difference of simple}, \ref{combinatorics of sigma*}, and \ref{combinatorics of sigma* in refined}) are proved in Section \ref{section: combinatorics}. Finally, Section \ref{section: spt function} is devoted to the smallest-part analogues, where we prove Theorems \ref{sptEO thm}, \ref{sptOE thm}, and \ref{sptEOud thm}. We conclude the paper with some remarks and pose several questions for further study in Section \ref{section: concluding remarks}.

\section{Preliminaries}\label{section: preliminaries}

The following definition of the basic hypergeometric series ${}_{r+1}\phi_r$ will be used throughout the paper:
\begin{align*}
	\pFq{r+1}{r}{a_1,a_2,\cdots,a_{r+1}}{b_1,b_2,\cdots,b_{r}}{q,z}:=\sum_{n=0}^\infty\frac{(a_1;q)_n(a_2;q)_n\cdots(a_{r+1};q)_n}{(q;q)_n(b_1;q)_n(b_2;q)_n\cdots(b_{r};q)_n}z^n.
\end{align*}

We will be needing the Euler--Maclaurin summation formula for two-dimensional complex-valued functions obtained by Bringmann, Jennings-Shaffer and Mahlburg \cite[Proposition 5.2]{bjm}. To state the formula, let us fix notations. We write $\overrightarrow{x}=(x_1,x_2)\in\mathbb{R}^2$. We say that a function $g:D\rightarrow\mathbb{C}^2$, where $D\subset\mathbb{C}^2$ is an unbounded domain, is of sufficient decay if there
exists $\epsilon_1>0$ and $\epsilon_2>0$ such that 
\begin{align*}
g(\overrightarrow{x})\ll\frac{1}{|x_1|^{1+\epsilon_1}|x_2|^{1+\epsilon_2}},
\end{align*}
uniformly  as $|x_1|+|x_2|\to\infty$ in $D$. Let us also define
\begin{align*}
D_\theta:=\{re^{i\alpha}:r\geq0\ \mathrm{and}\ |\alpha|\leq\theta\}, \quad 0\leq\theta<\frac{\pi}{2}.
\end{align*}
Bringmann, Jennings-Shaffer and Mahlburg \cite[Proposition 5.2]{bjm} proved the following summation formula.
\begin{proposition}\label{EMS}
	Let $0 \leq \theta < \frac{\pi}{2}$ and $g$ be a holomorphic function in a domain containing $D_\theta \times D_\theta$. Also assume that $g$ and all of its derivatives are of sufficient decay. Then for $a \in \mathbb{R}_{\geq 0}^2$, $N \in \mathbb{N}_0$, and as $z \to 0$ uniformly in $D_\theta$, we have
	\begin{align*}
		\sum_{m_1,m_2 \geq 0} g((\overrightarrow{m}+\overrightarrow{a})z) &= \frac{1}{z^2} \int_0^\infty \int_0^\infty g(\overrightarrow{x}) \, dx_1 dx_2 - \frac{1}{z} \sum_{n_1=0}^N \frac{B_{n_1+1}(a_1) z^{n_1}}{(n_1+1)!} \int_0^\infty g^{(n_1, 0)}(0, x_2) \, dx_2 \\
		&\quad - \frac{1}{z} \sum_{n_2=0}^N \frac{B_{n_2+1}(a_2) z^{n_2}}{(n_2+1)!} \int_0^\infty g^{(0, n_2)}(x_1, 0) \, dx_1 \\
		&\quad + \sum_{n_1+n_2 < N} \frac{B_{n_1+1}(a_1) B_{n_2+1}(a_2) f^{(n_1, n_2)}(\overrightarrow{0})}{(n_1+1)!(n_2+1)!}z^{n_1+n_2} + O(z^N),
	\end{align*}
where $B_n(x)$ are the Bernoulli polynomials.
\end{proposition}

We also require the following modified version of Ingham’s Tauberian theorem derived by Bringmann, Jennings-Shaffer and Mahlburg  \cite[Theorem 1.1]{bjm}. 
\begin{proposition}\label{ITT}
	Let $\displaystyle B(q) = \sum_{n=0}^\infty b(n)q^n$ be a power series whose radius of convergence is at least one and assume that $b(n)$ are non-negative and weakly increasing. Also suppose that $\lambda, \beta, \gamma \in \mathbb{R}$ with $\gamma > 0$ exist such that
	\begin{equation}
		B(e^{-t}) \sim \lambda t^{\beta} e^{\frac{\gamma}{t}} \quad \text{as } t \to 0^+, \quad B(e^{-z}) \ll |z|^{\beta} e^{\frac{\gamma}{\text{Re } z}} \quad \text{as } z \to 0, \nonumber
	\end{equation}
	with the latter condition holding in each region $R_{\Delta}$:
	\begin{equation}\label{r delta}
		R_{\Delta} := \{x + iy : x, y \in \mathbb{R}, x > 0, |y| \leq \Delta x\},
	\end{equation}
	 for $\Delta \geq 0$. Then we have
	\begin{equation*}
		b(n) \sim \frac{\lambda \gamma^{\frac{\beta}{2} + \frac{1}{4}}}{2\sqrt{\pi} n^{\frac{\beta}{2} + \frac{3}{4}}} e^{2\sqrt{\gamma n}} \quad \text{as } n \to \infty.
	\end{equation*}
\end{proposition}

\section{Partition function with parts separated by parity}\label{section: partitions with parity}

	\begin{proof}[\textbf{Thoerem \textup{\ref{oebar thm}}}][]
From \eqref{oe bar generating}, it is easy to see that
	\begin{align*}
		\overline{\mathrm{oe}}(q)&=\frac{1}{q(q^4;q^4)_\infty}\sum_{n=1}^{\infty}\frac{q^{2n}(q^4;q^4)_{n-1}}{(q^2;q^4)_n}\\
	&=\frac{q}{(q^4;q^4)_\infty(1-q^2)}\sum_{n=0}^{\infty}\frac{q^{2n}(q^4;q^4)_{n}}{(q^6;q^4)_{n}},
\end{align*}
where the last step follows upon replacing $n$ by $n+1$. We now make use of Heine's transformation \cite[p.~359, Equation (III.1)]{gr}
\begin{align}\label{hiene}
\sum_{n=0}^{\infty}\frac{(a)_n(b)_n}{(q)_n(c)_n}t^n=\frac{(b)_\infty(at)_\infty}{(c)_\infty(t)_\infty}\sum_{n=0}^{\infty}\frac{(\frac{c}{b})_n(t)_n}{(q)_n(at)_n}b^n,
\end{align}
with by first replacing $q$ by $q^4$ and then letting $a=q^4$, $t=q^2$, $b=q^4$ and $c=q^6$ to arrive at
	\begin{align*}
	\overline{\mathrm{oe}}(q)&=\frac{q}{(q^4;q^4)_\infty(1-q^2)}\left[ \frac{(q^4;q^4)_\infty(q^6;q^4)_\infty}{(q^6;q^4)_\infty(q^2;q^4)_\infty}\sum_{n=0}^{\infty}\frac{(q^2;q^4)_n(q^2;q^4)_n}{(q^4;q^4)_n(q^6;q^4)_n}q^{4n}\right].
	\end{align*}
Now invoking the ${}_2\phi_1$-transformation \cite[p.~359, Equation (III.4)]{gr}
\begin{align}\label{2phi1}
{}_2\phi_1(a,b;c;q,z)=\frac{(az;q)_\infty}{(z;q)_\infty}\sum_{n=0}^\infty\frac{(a;q)_n(c/b;q)_n(-bz)^nq^{\frac{n(n-1)}{2}}}{(q;q)_n(c;q)_n(az;q)_n},
\end{align}
with first replacing $q$ by $q^4$ and then letting $a=q^2$, $b=q^2$, $c=q^6$ and $z\rightarrow q^4$, we are led to
	\begin{align*}
	\overline{\mathrm{oe}}(q)&=\frac{q}{(1-q^2)(q^2;q^4)_\infty}\frac{(q^6;q^4)_\infty}{(q^4;q^4)_\infty}\sum_{n=0}^{\infty}\frac{(q^2;q^4)_n(q^4;q^4)_n}{(q^4;q^4)_n(q^6;q^4)_n(q^6;q^4)_n}q^{6n}(-1)^nq^{2n(n-1)}\\
		&=\frac{q}{(1-q^2)(q^4;q^4)_\infty}\sum_{n=0}^{\infty}\frac{(-1)^nq^{2n^2+4n}}{(1-q^{4n+2})(q^6;q^4)_n}\\
		&=\frac{1}{q(1-q^2)(q^4;q^4)_\infty}\sum_{n=0}^{\infty}\frac{(-1)^nq^{2(n^2+2n+1)}}{(1-q^{4n+2})(q^6;q^4)_n}\\
		&=\frac{1}{q(q^4;q^4)_\infty}\sum_{n=0}^{\infty}\frac{(-1)^nq^{2(n+1)^2}}{(1-q^{4n+2})(q^2;q^4)_{n+1}},
	\end{align*}
The required identity now follows upon replacing $n$ by $n-1$ and invoking the definition of $\mathcal{U}(q)$ from \eqref{u defn}.
\end{proof}

\begin{proof}[\textbf{Corollary \textup{\ref{hurwitz combinatorics}}}][]
The result follows from Theorem \ref{combinatorics of U2} and the fact that \cite[Equation (1.5)]{abl}
\begin{align*}
\mathcal{U}(q)=\sum_{n=1}^\infty H(8n-1)q^n.
\end{align*}
\end{proof}

\begin{proof}[\textbf{Theorem \textup{\ref{sigma* as generating function}}}][]
Note that
\begin{align}\label{before cauchy}
	\mathcal{K}(q)&=\sum_{n=0}^{\infty}\frac{q^{2n+3}}{(1-q^{2n+3})(q^2;q^2)_n} \nonumber\\
	&=q^3\sum_{n=0}^{\infty}\frac{q^{2n}}{(q^2;q^2)_n}\sum_{m=0}^{\infty}q^{(2n+3)m} \nonumber\\
	&=q^3\sum_{m=0}^{\infty}q^{3m}\sum_{n=0}^{\infty}\frac{q^{2n(m+1)}}{(q^2;q^2)_n}.
\end{align}
Employing the following identity of Cauchy \cite[p.~17, Theorem~2.1]{andrews}
\begin{align*}
	\sum_{n=0}^{\infty}\frac{(a;q)_n}{(q;q)_n}z^n=\frac{(az;q)_\infty}{(z;q)_\infty},
\end{align*}
for the inner sum of \eqref{before cauchy}, we are led to 
\begin{align}\label{before exercise}
	\mathcal{K}(q)=\frac{q^3}{(q^2;q^2)_\infty}\sum_{m=0}^{\infty}q^{3m}(q^2;q^2)_m.
\end{align}
Now using the following identity \cite[p.~29, Exercise (6)]{andrews}
\begin{align*}
\sum_{n=0}^{\infty}\left(\frac{-xq}{y};q^2\right)_ny^n=\sum_{n=0}^{\infty}\frac{q^{n^2}x^n}{(y;q^2)_{n+1}},
\end{align*}
(with replacing $y=q^3$ and $x=-q^4$) in \eqref{before exercise}, we obtain
\begin{align*}
	\mathcal{K}(q)&=\frac{q^3}{(q^2;q^2)_\infty}\sum_{n=0}^{\infty}\frac{(-1)^nq^{n^2+4n}}{(q^3;q^2)_{n+1}}\\
	&=\frac{1-q}{q(q^2;q^2)_\infty}\sum_{n=0}^{\infty}\frac{(-1)^nq^{(n+2)^2}}{(q;q^2)_{n+2}}\\
	&=\frac{1-q}{q(q^2;q^2)_\infty}\sum_{n=2}^{\infty}\frac{(-1)^nq^{n^2}}{(q;q^2)_{n}}\\
	&=\frac{1-q}{q(q^2;q^2)_\infty}\sum_{n=1}^{\infty}\frac{(-1)^nq^{n^2}}{(q;q^2)_{n}}+\frac{1}{(q^2;q^2)_\infty}.
\end{align*}
Invoking the definition of $\sigma^*(q)$ from \eqref{sigma* as generating function}, we complete the proof of the theorem.
\end{proof}

\begin{proof}[\textbf{Theorem \textup{\ref{sigma* representations}}}][]
Employing \eqref{hiene} with first replacing $q$ by $q^2$ and then letting  $t=q^2$, $b=q$, $c=q^3$ and $a=0$,  we have	
	\begin{align*}
		\sum_{n=0}^{\infty}\frac{(q;q^2)_n}{(q^3;q^2)_n(q^2;q^2)_n}q^{2n}=\frac{1-q}{(q^2;q^2)_\infty}\sum_{n=0}^{\infty}(q^2;q^2)_nq^{n}.
	\end{align*}
Multiplying both sides of the above equation by $q$ and rearranging expressions, we have
	\begin{align*}
	\sum_{n=0}^{\infty}(q^2;q^2)_nq^{n+1}=\sum_{n=0}^{\infty}\frac{q^{2n+1}}{(1-q^{2n+1})}(q^{2n+2};q^2)_\infty.
\end{align*}
The result now follows upon invoking \eqref{sigma* repres}.

We now prove \eqref{representation 2}. To that end, replacing $n$ by $n+1$ in \eqref{sigma*}, we can see that
\begin{align}\label{before fine}
	\sigma^*(q)&=-\sum_{n=0}^{\infty}\frac{(-1)^nq^{(n+1)^2}}{(q;q^2)_{n+1}}\nonumber\\
	&=-\frac{q}{(1-q)}\sum_{n=0}^{\infty}\frac{(-1)^nq^{n^2+2n}}{(q^3;q^2)_n}.	
\end{align}
We need the following transformation \cite[p.~4, Equation (6.1)]{fine},
$$F(a,0;t:q)=\frac{1}{1-t}\sum_{n=0}^{\infty}\frac{(-at)^nq^{\frac{n^2+n}{2}}}{(tq)_n},$$
where,
$$F(a,b;t:q)=1+\sum_{n=1}^{\infty}\frac{(aq)_n}{(bq)_n}t^n.$$
Replacing $q$ by $q^2$ and letting $t=q$, $a=1$ in the above transformation, we get
\begin{align}\label{fine1}
F(1,0,q:q^2)=\frac{1}{1-q}\sum_{n=0}^{\infty}\frac{(-1)^nq^{n^2+n+n}}{(q^3;q^2)_n}.
\end{align}
Equations \eqref{before fine} and \eqref{fine1} together imply that
\begin{align}
	\sigma^*(q)&=-qF(1,0,q:q^2).\nonumber
\end{align}
Now applying the transformation \cite[p.~3, Equation (4.3)]{fine}
\begin{align}
	F(a,b;t:q)=\frac{1}{1-t}\left(1+\frac{b-a}{(1-bq)}\sum_{n=1}^\infty\frac{(aq)_{n-1}}{(bq^2)_{n-1}}(tq)^n\right),\nonumber
\end{align}
(with replacing $q$ by $q^2$ and $a=1$, $b=1$, $t=q$), we arrive at
\begin{align*}
	\sigma^*(q)&=-\frac{q}{1-q}\left(1-\sum_{n=1}^{\infty}(q^2;q^2)_{n-1}q^{3n}\right).
\end{align*}
Hence the proof.

\end{proof}

\subsection{Quantum $q$-series identity}\label{section: quantum series}
\begin{proof}[\textbf{Theorem \textup{\ref{quantum series}}}][]
	If $q$ is a primitive $N$th root of unity, then for $n\le N-1$, we have \cite[Equation (2.3)]{lovejoy}
\begin{align*}
	\sum_{n=0}^{N-1}(b;q)_nz^n=b^{N-1}\sum_{n=0}^{N-1}(b;q)_n(q/z;q)_n(z/b)^nq^{-n^2-n}.
\end{align*}
First changing the base  $q$ by $q^2$ and then letting $z=q^3$ and $b=q^2$ in the above identity, we deduce that
\begin{align}\label{basic1}
	\sum_{n=0}^{N-1}(q^2;q^2)_n q^{3n}&=q^{-2}\sum_{n=0}^{N-1}(q^2;q^2)_n(1/q;q^2)_nq^{-2n^2-n}\nonumber\\
&=q^{-2}\sum_{n=0}^{N-1}(-1)^n(q^{-2};q^{-2})_n(1/q;q^2)_nq^{-n^2},
\end{align}
where in the last step we used the following fact
\begin{align}\label{basic}
	(aq^{-1};q^{-1})_n=(q/a;q)_n(-a)^nq^{-\frac{n(n+1)}{2}},
\end{align}
with replacing $q$ by $q^2$ and letting $a=1$. Again employing \eqref{basic} but with replacing $q$ by $q^2$ and letting $a=q^3$ in \eqref{basic1}, we find that
\begin{align}\label{basic2}
\sum_{n=0}^{N-1}(q^2;q^2)_n q^{3n}=\sum_{n=0}^{N-1}(q^{-2};q^{-2})_n(q;q^{-2})q^{-2n-2}.
\end{align}
We now need the following special case of Sears transformation \cite[Equation (4.3)]{lovejoy}
\begin{align*}
	\sum_{n=0}^{N-1}\frac{(q^{1-N};q)_n(b;q)_n(c;q)_n}{(q;q)_n(d;q)_n}q^n=c^{N-1}\sum_{n=0}^{N-1}\frac{(-1)^n(q^{1-N};q)_n(d/b;q)_n(c;q)_n}{(q;q)_n(d;q)_n}\left(\frac{b}{c}\right)^nq^{Nn-\frac{n(n+1)}{2}}.
\end{align*}
Replacing $q$ by $q^2$, then setting $b=q^2$, $c=1/q$, and letting $d\to0$, and finally taking $q$ to be a primitive $N$th root of unity, we obtain
\begin{align}\label{basic3}
\sum_{n=0}^{N-1}(q^{2};q^{2})_n(q^{-1};q^{2})_nq^{2n}&=q	\sum_{n=0}^{N-1}(1/q;q^2)_n(-1)^nq^{-n^2-n}q^{3n}\nonumber\\
	&=	q\sum_{n=0}^{N-1}(q;q^{-2})_n,
\end{align}
where in last step we invoked \eqref{basic} with $q\rightarrow q^2$ and $a\rightarrow q^3$. Replacing $q$ by $q^{-1}$ in \eqref{basic3} and substitute the resultant in \eqref{basic2}, we are led to
\begin{align*}
\sum_{n=0}^{N-1}(q^{2};q^{2})_n q^{3n}&=q^{-3}\sum_{n=0}^{N-1}(q^{-1};q^{2})_n\\
&=q^{-3}\left(1+(1-q^{-1})\sum_{n=1}^{N-1}(q;q^{2})_{n-1}\right)\\
&=q^{-3}\left(1-\frac{1-q}{q}\sum_{n=1}^{N-2}(q;q^{2})_{n}\right).
\end{align*}
Taking $N\to\infty$ on both sides yields
\begin{align*}
\sum_{n=0}^{\infty}(q^{2};q^{2})_nq^{3n}&=q^{-3}\left(1-\frac{1-q}{q}\sum_{n=1}^{\infty}(q;q^{2})_{n}\right)\\
&=q^{-3}\left(1+\frac{1-q}{q}\sigma^*(q)\right),
\end{align*}
where the last step follows upon invoking the second identity from \eqref{quantum2} which is valid for $q$ being a primitive odd root of unity.  The required identity \eqref{quantum series eqn} now follows upon using \eqref{quantum1} and simplifying the expressions.
\end{proof}

\subsection{Asymptotics of partition functions}\label{section: asymptotics}

\begin{lemma}
	We have,
$$(1-q^2)\big(\mathcal{K}(q)+q^8\big)\succeq0.$$
Here, and through out the sequel, by $\displaystyle\sum_{n=0}^\infty a(n)q^n\succeq 0$, we mean $a(n)\geq0$ for all $n\geq0$.
\end{lemma}
 \begin{proof}
	Note that
	\begin{align*}
		&(1-q^2)\big(\mathcal{K}(q)+q^8\big)\\
		&=(1-q^2)\sum_{n=0}^{\infty}\frac{q^{2n+3}}{(1-q^{2n+3})(q^2;q^2)_n}	+q^8(1-q^2)\\
		&=\frac{(1-q^2)q^3}{(1-q^3)}+\sum_{n=1}^{\infty}\frac{q^{2n+3}}{(1-q^{2n+3})(q^4;q^2)_{n-1}}	+q^8(1-q^2)\\
		&=(q^3-q^5)(1+q^3+q^6+q^9+\dots)+\sum_{n=1}^{\infty}\frac{q^{2n+3}}{(1-q^{2n+3})(q^4;q^2)_{n-1}}	+q^8-q^{10}\\
		&=(q^3+q^6+q^9+\dots)-(q^5+q^8+q^{11}\dots)+\sum_{n=1}^{\infty}\frac{q^{2n+3}}{(1-q^{2n+3})(q^4;q^2)_{n-1}}	+q^8-q^{10}\\
		&=(q^3+q^6+q^9+\dots)+\left(\sum_{n=1}^{\infty}\frac{q^{2n+3}}{(1-q^{2n+3})(q^4;q^2)_{n-1}}-(q^{11}+q^{14}+\dots)\right)-q^5-q^{10}.
	\end{align*}		
	Note that the series $\sum_{n=1}^{\infty}\frac{q^{2n+3}}{(1-q^{2n+3})(q^4;q^2)_{n-1}}$ is the generating function for  $T(n)$, where $T(n)$ denotes the number of partitions of $n$ having largest part odd and greater than $3$, while all remaining parts are even  (greater than or equal to $4$), and no part is consecutive to the largest odd part. It can be checked that $T(5)=1=T(10)$, and hence the coefficients of $q^5$ and $q^{10}$ on the right-hand side are zero. Moreover, every exponent of $q$ occurring in
	the expression $$(q^{11}+q^{14}+q^{17}+\dots)$$  is of the form $(3n+5)$, where $n\ge 2$. Therefore, in order to prove $(1-q^2)\big(\mathcal{K}(q)+q^8\big)\succeq 0$, it is enough to show that $T(3n+5)\ge 1$ for $n\ge2$. To this end, we exhibit, for every $n\geq2$, at least one partition of $3n+5$ counted by $T(3n+5)$. We have,
	$$3n+5 = 
	\begin{cases}  
		(2n+3)+(n+2), & \text{if } \text{$n$ is even}, \\
		(3k+4)+(3k+4), & \text{if } n=2k+1\ \text{and $k$ is odd}, \\
		{5+5+5+5}, & \text{if } n=2k+1\  \text{and}\  k=2,\\
		(7+7+7+7)+\underbrace{(4+4+4+\dots)}_{(3\ell-5)\ \text{copies}}, & \text{if } n=2k+1 \ \text{and}\  k=2\ell\ge 4. \end{cases}
		$$
	Each of the above partitions satisfies the defining conditions for $T(n)$. Hence $T(3n+5)\geq1,\ n\geq2$, which  completes the proof.
\end{proof}

We now determine the first two terms in the asymptotic expansion of $\sigma^*(-e^{-z})$ as $z\to 0$.

\begin{lemma}\label{asymptotic of sigma*}
	Let $q = e^{-z}$ and $\Delta\ge0$. Then, as $z\rightarrow 0$ in $R_\Delta$, we have
	\begin{align}\label{asymptotic of sigma* eqn}
		\sigma^*(-q)\sim 1-3z,
	\end{align}
	and 
		\begin{align}\label{asymptotic of sigma* eqn 2}
		\sigma^*(q)\sim -1.
	\end{align}
\end{lemma}
\begin{proof}
We start with the following representation of $\sigma^*(q)$, given by Andrews, Dyson and  Hickerson \cite[Theorem 6]{adh},
\begin{align}
\sigma^{*}(-q)&=\sum_{n\ge 1}^{}(-1)^nq^{n(3n-1)}(1+q^{2n})\sum_{j=0}^{2n-1}(-q)^{-\frac{j(j+1)}{2}}.\nonumber
\end{align}
Note that the above expression can be recast as
\begin{align}\label{a1+a2}
\sigma^{*}(-q)&=A_1(q)+A_2(q),
\end{align}
where 
\begin{align}
A_1(q):=\sum_{n\ge 1}^{}(-1)^nq^{n(3n-1)}(1+q^{2n})\sum_{j=0}^{n-1}(-1)^{{j(2j+1)}}q^{-{j(2j+1)}},\nonumber
\end{align}
and 
\begin{align}
A_2(q):=\sum_{n\ge 1}^{}(-1)^nq^{n(3n-1)}(1+q^{2n})\sum_{j=0}^{n-1}(-1)^{{(j+1)(2j+1)}}q^{-{(j+1)(2j+1)}}.\nonumber
\end{align}
Rewriting $A_1(q)$ further, we obtain 
\begin{align}
A_1(q)&=q^{\frac{1}{24}}\sum_{n\ge 1}^{}(-1)^nq^{3\big(n-\frac{1}{6}\big)^2}\sum_{j=0}^{n-1}(-1)^{{j(2j+1)}}q^{-{2(j+\frac{1}{4})^2}}\nonumber\\
&\quad+q^{\frac{1}{24}}\sum_{n\ge 1}^{}(-1)^nq^{3\big(n+\frac{1}{6}\big)^2}\sum_{j=0}^{n-1}(-1)^{{j(2j+1)}}q^{-{2(j+\frac{1}{4})^2}}\nonumber\\
&=q^{\frac{1}{24}}\sum_{n,j\ge 0}^{}(-1)^{n+1}\left[q^{3\big(n+\frac{7}{12}\big)^2+\big(j+\frac{1}{4}\big)^2+6\big(n+\frac{7}{12}\big)\left(j+\frac{1}{4}\right)}+q^{3\big(n+\frac{11}{12}\big)^2+\big(j+\frac{1}{4}\big)^2+6\big(n+\frac{11}{12}\big)\big(j+\frac{1}{4}\big)}\right],\nonumber\\
&=q^{\frac{1}{24}}\sum\limits_{\substack{r\in \{0,1\} \\ \delta\in \{3,5\}}}^{}(-1)^{r+1}\sum_{n,j\ge 0}^{} g\left(\left(n+\frac{12r+2\delta+1}{24},j+\frac{1}{4}\right)\sqrt{z}\right),\nonumber
\end{align}
where the penultimate step follows upon replacing $n$ by $n+j+1$ and simplifying expressions. Here the function $g(x,y)$ is defined as $g(x,y):=e^{-12x^2-y^2-12xy}$. Now we apply Proposition \ref{EMS} with the above function $g(x,y)$ and letting $N=3$ and performing the necessary simplifications, we deduce that
\begin{align}\label{A1 final}
A_1(q)=e^{-\frac{z}{24}}\left(\frac{1}{2}-\frac{1}{2}\sqrt{\frac{\pi}{z}}+\frac{23}{48}\sqrt{\pi z}-\frac{71}{48}z+O(z^{\frac{3}{2}})\right).
\end{align}
Similarly, we can evaluate $A_2(q)$ as
\begin{align}\label{A2 final}
A_2(q)=e^{-\frac{z}{24}}\left(\frac{1}{2}+\frac{1}{2}\sqrt{\frac{\pi}{z}}-\frac{23}{48}\sqrt{\pi z}-\frac{71}{48}z+O(z^{\frac{3}{2}})\right).
\end{align}
Substituting \eqref{A1 final} and \eqref{A2 final} in \eqref{a1+a2}, expanding the exponential functions, and simplifying the resulting expressions, we arrive at \eqref{asymptotic of sigma* eqn}.

Equation \eqref{asymptotic of sigma* eqn 2} follows from the same argument as above.
\end{proof}

\begin{proof}[\textbf{Theorem \textup{\ref{asymptotic of kn}}}][]
Let us define
	\begin{align*}
		\mathcal{K}_{e}(q):=\sum_{n=1}^{\infty}k(2n)q^{2n}=\frac{\mathcal{K}(q)+\mathcal{K}(-q)}{2},
	\end{align*}
	and 
		\begin{align*}
		\mathcal{K}_{o}(q):=\sum_{n=1}^{\infty}k(2n-1)q^{2n-1}=\frac{\mathcal{K}(q)-\mathcal{K}(-q)}{2}.
	\end{align*}
Invoking Theorem \ref{sigma* as generating function}, we obtain
	\begin{align}\label{keqm}
		\mathcal{K}_{e}(q)&=\frac{1}{2(q^2;q^2)_\infty}\left(\frac{(1-q)}{q}\sigma^{*}(q)-\frac{(1+q)}{q}\sigma^{*}(-q)+2	\right).
	\end{align}
An application of Lemma \ref{asymptotic of sigma*} implies that
\begin{align}\label{keq1}
	\frac{1-q }{q}\sigma^*(q)\sim -z,\ \ \ \ \ \frac{1+q }{q}\sigma^*(-q)\sim 2-5z.
\end{align}
Employing \eqref{keq1} along with the fact
\begin{align*}
\frac{1}{(q^2;q^2)_\infty}\sim \sqrt{ \frac{z}{\pi}}e^{\frac{\pi^2}{12z}},
\end{align*}
in \eqref{keqm}, we deduce that
\begin{align*}
\mathcal{K}_{e}(q)\sim 2z\sqrt{\frac{z}{\pi}}e^{\frac{\pi^2}{12z}}.
\end{align*}
This implies that
\begin{align*}
\mathcal{K}_{e}(q^{\frac{1}{2}})\sim z\sqrt{\frac{z}{2\pi}}e^{\frac{\pi^2}{6z}}.
\end{align*}
In a similar way, we can prove that 
\begin{align*}
q^{-\frac{1}{2}}	\mathcal{K}_{o}(q^{\frac{1}{2}})\sim\frac{1}{ (q;q)_\infty}.
\end{align*}
Now note that
$$\mathcal{K}_{e}(q^{\frac{1}{2}})=\sum_{n=1}^{\infty} k(2n)q^{n},\qquad
\mathrm{and}\qquad q^{-\frac{1}{2}}\mathcal{K}_{o}(q^{\frac{1}{2}})=\sum_{n=1}^{\infty} k(2n-1)q^{n}.$$
Both of the required asympotitc expansions now follow upon invoking Proposition \ref{ITT}, first with $\lambda=\frac{1}{\sqrt{2\pi}}$, $\beta=\frac{3}{2}$,  and $\gamma=\frac{\pi^2}{6}$ and then with $\lambda=\frac{1}{\sqrt{2\pi}}$, $\beta=\frac{1}{2},$ and $\gamma=\frac{\pi^2}{6}.$

\end{proof}

\begin{proof}[\textbf{Corollary \textup{\ref{Len asymptotic}}}][]
	First of all, from \eqref{Le generating}, it is easy to see that
	\begin{align*}
(1-q^2)(-q;q)_\infty\mathcal{S}_{\mathcal{EO}_{ud}}(-q)\succeq 0.
\end{align*}
This implies that $\mathcal{L}_{e}(n)$ is a weakly increasing sequence. 

Let $q=e^{-z}$. As $z\to0$, we need the following fact that \cite[p.~952, Equation (25)]{Zagier}
\begin{align*}
	H(q)\sim L(-1,\chi),
\end{align*} 
where $L(s,\chi)=\sum_{n\geq1}\chi(n)/n^{s}$ is the Dirichlet $L$-function associated with the unique primitive character $\chi(.)$ of conductor $12$. Note that $L(-1,\chi)=-2$. This can be seen by using the well-known results \cite{iw},  \cite[Proposition 4.1]{wash}
\begin{align*}
L(1-k,\chi)=-\frac{B_{k,\chi}}{k}, \qquad \mathrm{and}\qquad B_{k,\chi}=f^{k-1}\sum_{a=1}^f \chi(a)B_n\left(\frac{a}{f}\right),
\end{align*}
where $f$ is the conductor, $B_n$ are the classical Bernoulli polynomials and  $B_{k,\chi}$ are generalized Bernoulli polynomials:
\begin{align*}
\sum_{a=1}^f\frac{\chi(a) t e^{at}}{e^{ft}-1}=\sum_{n=0}^\infty B_{n,\chi}\frac{t^n}{n!}.
\end{align*}
We also need \cite[p.~163, Theorem 7.15]{titch}
$$\sum_{n=1}^\infty d(n)q^n \sim \frac{1}{z} (\gamma-\log z) + \frac{1}{4},$$
and
\begin{align*}
	(q;q)_\infty\sim \sqrt{\frac{2\pi}{z}}e^{-\frac{\pi^2}{6z}} \qquad \mathrm{and}
\qquad	\frac{1}{(q;q^2)_\infty}\sim {\frac{1}{\sqrt{2}}}e^{\frac{\pi^2}{12z}}.
\end{align*}
Using all these results in Theorem \ref{sptEOud thm} with letting $q=e^{-t}$,  we deduce that
\begin{align}\label{last}
(-q;q)_\infty\mathcal{S}_{\mathcal{EO}_{ud}}(-q)\sim \frac{1}{2}\sqrt{\frac{z}{2\pi}}e^{\frac{\pi^2}{6z}}.
\end{align}
Now the result immediately follows from \eqref{Le generating}, \eqref{last} and employing Proposition \ref{ITT} with $\beta=\frac{1}{2}$, $\lambda=\frac{1}{2\sqrt{2\pi}}$ and $\gamma= \frac{\pi^2}{6}$.
\end{proof}

\subsection{Combinatorics of $\mathcal{U}(q)$ and $\sigma^*(q)$ in terms of odd Ferrers graph}\label{section: combinatorics}
\begin{proof}[\textbf{Theorem \textup{\ref{combinatorics of U2}}}][]
Let us take the following example of self conjugate colored Durfee odd Ferrers graph:
\begin{align*}
\begin{array}{r *{9}{c}} 
	&	\textcolor{blue}{2}& 	\textcolor{blue}{2}& 	\textcolor{blue}{2}& 	\textcolor{blue}{1}&  	\textcolor{blue}{1}& 	\textcolor{red}{1} & 	\textcolor{blue}{1} &	\textcolor{blue}{1}&  \\
	\cline{3-4}
	&		\textcolor{blue}{2}& 	\multicolumn{1}{|c}{	\textcolor{blue}{2}} & \multicolumn{1}{c|}{	\textcolor{blue}{2}} & 	\textcolor{blue}{2}  &  	\textcolor{blue}{2} & 	\textcolor{red}{2}  &  	\textcolor{blue}{2} &   \\
	&		\textcolor{blue}{2}& \multicolumn{1}{|c}{	\textcolor{blue}{2}} & \multicolumn{1}{c|}{	\textcolor{blue}{2}} & 	\textcolor{blue}{2} &  	\textcolor{blue}{2}&  	\textcolor{red}{2} &   &   \\
	\cline{3-4} 
	&		\textcolor{blue}{1} & 	\textcolor{blue}{2} & 	\textcolor{blue}{2} &  &   &   &   &   \\
	&		\textcolor{blue}{1} & 	\textcolor{blue}{2}& 	\textcolor{blue}{2} &  &   &   &   &   \\
	&		\textcolor{red}{1} & 	\textcolor{red}{2} & 	\textcolor{red}{2} &  &   &   &   &   \\
	&		\textcolor{blue}{1} & 	\textcolor{blue}{2} &  &   &   &   &   &   \\
	&		\textcolor{blue}{1}&  &  &   &   &   &   &   
\end{array}
\end{align*}

If the Durfee square has size $(n-1)$, then it contributes $2(n-1)^2$, while its border contributes $(4n-2)$, that is, Durfee square together with its border contributes a factor of $q^{2n^2}$. The rows below the Durfee square can be interpreted as partitions into odd parts, each less than or equal to $(2n -1)$, where the part $(2n-1)$, if present, may appear in two colors. Moreover, the columns to the right of the Durfee square can likewise be interpreted as an arbitrary partition into odd parts less than or equal to $(2n-1)$,  where the part $(2n-1)$, if present, may appear in two colors. Since the graph is self-conjugate, the parts to the right and below the Durfee square are generated by $$\frac{1}{(1-q^{4n-2})(q^2;q^4)_n}.$$ Hence, $\mathcal{U}(q^2)$ is the generating function for excess number of self conjugate colored Durfee odd Ferrers graphs with even size Durfee square  over odd size Durfee square.
\end{proof}

\begin{remark}
	Self-conjugate colored Durfee odd Ferrers graph can be generated by
 $
 	\sum_{n=1}^\infty\frac{q^{2n^2}}{(1-q^{4n-2})(q^2;q^4)_n}.
$
 Moreover, observe that it counts the total number of largest parts present in all the partitions counted by $c(n)$, where $c(n)$ denotes the number of partitions of $n$ into consecutive odd parts with even multiplicities.
\end{remark}

\begin{proof}[\textbf{Theorem \textup{\ref{U in terms of difference of odd Ferrers}}}][]
We first prove \eqref{new representation of U}. This follows as a direct application of ${}_2\phi_1$-transformation given in \eqref{2phi1} with	first replacing $q$ by $q^2$ and then letting $a= q^2$, $z=q^m$, $b=q^2$ and $c\rightarrow q^3$.

We now show the combinatorial interpretation of $\mathcal{U}(q)$.

	Consider the following Refined odd Ferrers graph of the partition $ 10+8+7+4+3+3+1$:
\begin{align*}
	\overbrace{
		\begin{array}{r *{10}{c}} 
           &	1& 	1 & 1& 1 & 1 & 1 & 1 & 1&1&1\\
			&	0& 	2&2 & 2  &  2 &   &   &  & &   \\
			&	1& 2 & 2& 2  &  &   &   &   & &\\
			&	 0& 2 & 2 &  &   &   &   &   & &\\
			&	1 & 2 &  &   &   &   &   &   & &\\
			&	1 &  2&  &   &   &   &   &   & &\\
			&	1 &  &   &   &   &   &   &   & &\\
	\end{array}}^{\text{$n$ number of 1's}}
\end{align*}
Let the first row contain $n$ ones. Note that the remaining rows which contain zeros correspond to even parts. Since no such row can be repeated, the corresponding even parts are distinct. Hence, the rows below the first row can be interpreted as a partition with parts less than or equal to $2n-1$ and having even parts distinct. Hence, 
$$\mathcal{U}(q)=\sum_{n=1}^{\infty}\frac{(q^2;q^2)_{n-1}}{(q;q^2)_n}q^n$$ 
is the generating function for excess number of Refined odd ferror graphs with even number of 0's on boundary over odd number of $0$'s on boundary.
\end{proof}

\begin{proof}[\textbf{Theorem \textup{\ref{U in terms of difference of simple}}}][]
Note that the term $$\frac{1}{(1-q^{2n-1})(q;q^2)_n}=\frac{1}{(1-q^{2n-1})^2(q;q^2)_{n-1}}$$ 
is the generating function for partitions into two colors with odd parts less than or equal to $(2n-1)$, where only the largest part may appear in both colors. Moreover,  $q^{n^2}=q^{1+3+\dots+(2n-1)}$ guarantees that each of the odd parts $1,3,\dots,(2n-1)$ occurs at least once.  Clearly, the number of such distinct odd parts is $n$. Therefore, combining these observations, we conclude that  $$\mathcal{U}(q)=-\sum_{n=0}^{\infty}\frac{(-1)^{n}q^{n^2}}{(1-q^{2n-1})(q;q^2)_n}$$ is the generating function for $w_o(n)-w_e(n)$.
\end{proof}

	\begin{proof}[\textbf{Theorem \textup{\ref{combinatorics of sigma*}}}][]
	Consider the following self conjugate Durfee odd Ferrers graph of the partition $10+10+8+5+3+1$:
	\begin{align*}
		\begin{array}{r *{8}{c}} 
			&	2& 	2 & 2& 1 & 1 & 1 & 1 & \\
			\cline{3-4}
			&	2& 	\multicolumn{1}{|c}{2} & \multicolumn{1}{c|}{2} & 2  &  2 &   &   &   \\
			&	2& \multicolumn{1}{|c}{2} & \multicolumn{1}{c|}{2} & 2  &  &   &   &   \\
			\cline{3-4} 
			&	1 & 2 & 2 &  &   &   &   &   \\
			&	1 & 2 &  &   &   &   &   &   \\
			&	1 &  &  &   &   &   &   & 
		\end{array}
	\end{align*}
As we discussed in the proof of Theorem \ref{combinatorics of U2}, a Durfee square of size $(n-1)$ together with its border contributes a factor of $q^{2(n-1)^2+4n-2}$.	The rows below the Durfee square can be interpreted as partitions into odd parts each less than or equal to $(2n -1)$. Similarly, if we read the columns to the right of the Durfee square, we obtain another arbitrary partition into odd parts each less than or equal to $(2n-1)$. Since the graph is self conjugate, the combined contribution from the rows below and the columns to the right of the Durfee square is $\frac{1}{(q^2;q^4)_n}$. Hence $\sigma^*(q^2)$ is the generating function for excess number of self conjugate Durfee odd Ferrers graphs with odd size Durfee square over even size Durfee square.
\end{proof}

\begin{proof}[\textbf{Theorem \textup{\ref{combinatorics of sigma* in refined}}}][]
	Consider the following Refined odd Ferrers graph 7+8+6+4+2
	\begin{align*}
	\overbrace{
		\begin{array}{r *{8}{c}} 
			&	1& 	1 & 1& 1 & 1 & 1 & 1 & \\
			&	0& 	2&2 & 2  &  2 &   &   &   \\
			&	0& 2 & 2& 2  &  &   &   &   \\
			&	 0& 2 & 2 &  &   &   &   &   \\
			&	0 & 2 &  &   &   &   &   &   
	\end{array}}^{ \text{$(n+1)$ number of 1's}}
	\end{align*}
	Let the first row contain $(n+1)$ entries, all of which are $1$'s.
	Clearly, all the entries in the first column, except the first, are zero. Moreover, all rows except the first represent an even part less than or equal to $2n$. Hence, $\sigma^*(q)=-\sum_{n=1}^{\infty}(q^2;q^2)_nq^{n+1}$ is the generating function for excess number of Refined odd Ferrers graphs where all the entries (except first) of first column are zero with odd in counting over even in counting.
\end{proof}

\subsection{Partition identities arising from Combinatorics of $\sigma^*(q)$ and $\mathcal{U}(q)$}

We have provided two new  combinatorial interpretations of the coefficients of $\sigma^*(q)$ in  Theorem \ref{combinatorics of sigma* in refined} and Corollary \ref{sigma* combin}. Previously, Andrews, Dyson and Hickerson \cite[p.~404]{adh}, and  Andrews, Garvan and Liang \cite[Corollary 2.3]{agl} also obtained such combinatorial interpretations. More precisely \cite[Corollary 2.3]{agl}, 
\begin{align*}
\mathscr{B}_e(n)-\mathscr{B}_o(n)=\mathscr{C}_3(n)-\mathscr{C}_1(n),
\end{align*}
where $\mathscr{B}_e(n)$ \textup{(}resp. $\mathscr{B}_o(n)$\textup{)} denote the number of partitions of $n$ with an odd number of smallest parts, and with an even \textup{(}resp. odd\textup{)} number of parts. Consider partitions into odd parts without gaps, then for
$j = 1$ or $3$, let $\mathscr{C}_j(n)$ denote the number of such partitions of $n$ in which the largest part is congruent to $j$ mod $4$.

Our results in Theorem \ref{combinatorics of sigma* in refined} and Corollary \ref{sigma* combin} immediately extend the above result by adding two new partition identities.
\begin{corollary}\label{sigma* results}
	We have
\begin{align*}
\mathscr{B}_e(n)-\mathscr{B}_o(n)=\mathscr{C}_3(n)-\mathscr{C}_1(n)	=\mathscr{D}_o(n)-\mathscr{D}_e(n)=\mathscr{E}_o(n)-\mathscr{E}_e(n),
\end{align*}
where  $\mathscr{D}_e(n)$ \textup{(}resp. $\mathscr{D}_o(n)$\textup{)} denote number of Refined  odd Ferrers graphs of size $n$, where all entries \textup{(}except first\textup{)} of the first column are zero and even \textup{(}resp. odd\textup{)}  in counting.  Moreover, $\mathscr{E}_e(n)$ \textup{(}resp. $\mathscr{E}_o(n)$\textup{)} counts the number of partitions of $n$ having smallest part odd and other parts even and distinct with even \textup{(}resp. odd\textup{)} in counting.
\end{corollary}

We next record several partition identities that follow by combining our results about combinatorics of $\mathcal{U}(q)$ with identities previously established in the literature. To that end, we first recall the definition of M2-rank. To that end, we use the notation $\ell(\lambda)$ to denote the largest part of a partition, $\nu(\lambda)$ to denote the number of
parts, and $\lambda_0$ for the subpartition of an overpartition consisting of the odd nonoverlined parts. Then the $M2$-rank of an overpartition
$\lambda$ is given by
\begin{align*}
	\textup{M}2\textup{-rank}= \left\lceil {\frac{\ell(\lambda)}{2}}\right\rceil-\nu(\lambda)-\nu(\lambda_0)-\chi(\lambda),
\end{align*}
where $\chi(\lambda)=1$, if largest part of $\lambda$ is odd and non overlined and $\chi(\lambda)=0$ otherwise.

Our identities in \ref{hurwitz combinatorics eqn} and \eqref{U in terms of difference of odd Ferrers eqn}  along with the result of Bringmann and Lovejoy \cite[Corollary 1.1]{bl}, it is straightforward to see the following corollary.
\begin{corollary}\label{partition difference1}
	Let $\overline{P_e}(n)$ \textup{(}resp. $\overline{P_o}(n)$, $M2_e(n)$, $M2_o(n)$\textup{)} denote the number of overpartitions of $n$ with even rank \textup{(}resp. odd rank, even $M2$ rank, odd $M2$ rank\textup{)}. Then, we have
	\begin{align}
		v_{e}(n)-v_{o}(n)=u_{e}(2n)-u_{o}(2n)=\frac{1}{16}\big(\overline{P_e}(8n-1)-\overline{P_o}(8n-1)\big),\nonumber
	\end{align}
	and
	\begin{align}
		v_{o}(n)-v_{e}(n)=u_{o}(2n)-u_{e}(2n)=\frac{1}{8}\big(M2_e(8n-1)-M2_o(8n-1)\big).\nonumber
	\end{align}
\end{corollary}

The second order mock theta function $A(q)$ is defined as \cite[p.~285]{mc}
\begin{align*}
A(q):=\sum_{n=0}^\infty N_A(n)q^n=\sum_{n=0}^\infty\frac{q^{(n+1)^2}(-q;q^2)_n}{(q;q^2)_{n+1}^2}=\sum_{n=0}^\infty\frac{q^{n+1}(-q^2;q^2)_n}{(q;q^2)_{n+1}}.
\end{align*}
Chen and Garvan \cite[Lemma 3.1]{cg} obtained the following congruence:
\begin{align*}
N_A(n)\equiv (-1)^{n+1}H(8n-1)\qquad (\mathrm{mod}\ 4).
\end{align*}
This result along with \eqref{U in terms of difference of odd Ferrers eqn} immediately yields the following congruence.
\begin{corollary}\label{partition difference2}
	For each positive integer $n>0$, we have
	\begin{equation}
		v_{e}(n)-v_{o}(n)\equiv(-1)^{n+1}N_A(n)\pmod{4}.\nonumber
	\end{equation}
\end{corollary}

\section{Smallest part partition function analogues}\label{section: spt function}

\begin{proof}[\textbf{Theorem \textup{\ref{sptEO thm}}}][]
The definition of $	\mathcal{S}_\mathcal{EO}(q)$ from \eqref{sptEO} implies that
\begin{align*}
		\mathcal{S}_\mathcal{EO}(q)&=\sum_{n=0}^{\infty}\frac{q^{2n+1}}{(q^2;q^2)_n(q^{2n+1};q^2)_\infty(1-q^{2n+1})}\\
	&=\frac{q}{(q;q^2)_\infty}\sum_{n=0}^{\infty}\frac{q^{2n}(q;q^2)_n}{(q^2;q^2)_n(1-q^{2n+1})}.
\end{align*}
Now employing Watson’s $q$-analog  of Whipple’s theorem \cite{bailey}
\begin{align*}
	&{}_8\phi_7 \left[ 
	\begin{matrix} 
		a, q\sqrt{a}, -q\sqrt{a}, d, e, f, g, q^{-N} \\ 
		\sqrt{a}, -\sqrt{a}, \frac{aq}{d}, \frac{aq}{e}, \frac{aq}{f}, \frac{aq}{g}, aq^{N+1} 
	\end{matrix} 
	; q, \frac{a^2 q^{2+N}}{defg} 
	\right] \\
	&= \frac{(aq, aq/fg; q)_N}{(aq/f, aq/g; q)_N} 
	{}_4\phi_3 \left[ 
	\begin{matrix} 
		aq/de, f, g, q^{-N} \\ 
		aq/d, aq/e, fgq^{-N}/a 
	\end{matrix} 
	; q, q 
	\right],
\end{align*}
with replacing $q$ by $q^2$ and then letting $a=q^2$, $d=f=g=q$ and also letting $e\to\infty$ and $N\to\infty$, we obtain 
\begin{align}\label{before abl}
	\mathcal{S}_\mathcal{EO}(q)&=\frac{(q; q^2)_{\infty}}{(q^2; q^2)_{\infty}^2} \sum_{n=0}^{\infty}\frac{(1+q^{2n+1})}{\left(1 - q^{2n+1}\right)^2}q^{(n+1)(2n+1)} \nonumber\\
	&=\frac{(q; q^2)_{\infty}}{(q^2; q^2)_{\infty}^2}\left[\sum_{n=0}^{\infty}\frac{q^{n(n+1)}}{\left(1 - q^{2n+1}\right)^2}q^{(n+1)^2}+\sum_{n=0}^{\infty}\frac{q^{n(n+1)+2n+1}}{\left(1 - q^{2n+1}\right)^2}q^{(n+1)^2}\right] \nonumber\\
	&=\frac{(q; q^2)_{\infty}}{(q^2; q^2)_{\infty}^2}\left[\sum_{n=0}^{\infty}\frac{q^{n(n+1)}}{\left(1 - q^{2n+1}\right)^2}q^{(n+1)^2}+\sum_{n=1}^{\infty}\frac{q^{n^2+3n-2}}{\left(1 - q^{2n-1}\right)^2}q^{(n-1)^2}\right] \nonumber\\
	&=\frac{(q; q^2)_{\infty}}{(q^2; q^2)_{\infty}^2}	\left[\sum_{n=0}^{\infty}\frac{q^{n(n+1)}}{\left(1 - q^{2n+1}\right)^2}q^{(n+1)^2}+\sum_{n=1}^{\infty}\frac{q^{n(n-1)}q^{(n-1)^2}}{\left(1 - q^{-2n+1}\right)^2}\right] \nonumber\\
	&=\frac{(q; q^2)_{\infty}}{(q^2; q^2)_{\infty}^2}	\left[\sum_{n=0}^{\infty}\frac{q^{n(n+1)}}{\left(1 - q^{2n+1}\right)^2}q^{(n+1)^2}+\sum_{n=-\infty}^{1}\frac{q^{n(n+1)}q^{(n+1)^2}}{\left(1 - q^{2n+1}\right)^2}\right] \nonumber\\
	&=\frac{(q; q^2)_{\infty}}{(q^2; q^2)_{\infty}^2}\sum_{n \in \mathbb{Z}}\frac{q^{n(n+1)}}{\left(1 - q^{2n+1}\right)^2}q^{(n+1)^2}.
\end{align}
We now need the following relation of Alfes, Bringmann and Lovejoy \cite[Equation (2.2)]{abl}
\begin{align*}
	N^\circ(a, b, z; q) 
	&= \frac{(-aq;q^2)_\infty (-bq; q^2)_\infty}{(q^2;q^2)_\infty (abq^2; q^2)_\infty} \sum_{n \in \mathbb{Z}} \frac{(-q/a;q^2)_\infty (-q/b; q^2)_n (-ab)^n q^{n^2+3n+1}}{(1 - zq^{2n+1}) (-aq;q^2)_\infty (-bq; q^2)_{n+1}}.
\end{align*}
Letting $a=0$, $b=-1$ and $z=1$ in the above identity and substituting it in \eqref{before abl}, we get
\begin{align*}
	\mathcal{S}_\mathcal{EO}(q)=\frac{1}{(q^2; q^2)_{\infty}}	N^{0}(0, -1; 1; q).
\end{align*}	
The required result now follows upon invoking \eqref{N in terms of U1} in the above identity.
\end{proof}

\begin{remark}
We here note that \eqref{before abl} provides the following representation for $\mathcal{U}(q)$:
\begin{align*}
\mathcal{U}(q)=\frac{(q; q^2)_{\infty}}{(q^2; q^2)_{\infty}} \sum_{n \in \mathbb{Z}} \frac{ q^{2n^2+3n+1}}{\left(1 - q^{2n+1}\right)^2}.
\end{align*}
Using this representation, Bringmann, Lovejoy, and Osburn \cite[Lemma 6.1]{blo} construct a suitable completion of $\mathcal{U}(q)$ that transforms as a weak Maass form of weight $\frac{3}{2}$ on $\Gamma_0(16)$.

\end{remark}

\begin{proof}[\textbf{Theorem \textup{\ref{sptOE thm}}}][]
From \eqref{sptOE}, observe that
\begin{align*}
		\mathcal{S}_{\mathcal{OE}}(q)=\sum_{n=1}^{\infty}\frac{q^{2n}}{(1-q^{2n})(q;q^2)_n(q^{2n};q^2)_\infty}\\
	=	\frac{1}{(q^2;q^2)_\infty}	\sum_{n=1}^{\infty}\frac{q^{2n}(q^2;q^2)_{n-1}}{(1-q^{2n})(q;q^2)_n}.
\end{align*}
Upon replacing $n$ by $n+1$, we obtain
\begin{align}\label{before rlnb}
	\mathcal{S}_{\mathcal{OE}}(q)&=\frac{q^2}{(q^2;q^2)_\infty}	\sum_{n=0}^{\infty}\frac{q^{2n}(q^2;q^2)_{n}}{(1-q^{2n+2})(q;q^2)_{n+1}}\nonumber\\
	&=\frac{q^2}{(1-q^2)(q^2;q^2)_\infty(1-q)}	\sum_{n=0}^{\infty}\frac{q^{2n}(q^2;q^2)_{n}(q^2;q^2)_{n}}{(q^4;q^2)_n(q^3;q^2)_{n}}.
\end{align}
We now need the following $q$-series identity of Andrews  \cite[p.~141, Theorem 1]{andrews rlnb}
\begin{align*}
\sum_{n=0}^{\infty} \frac{(B;q)_{n}(-Abq;q)_{n} q^{n}}{(-aq;q)_{n}(-bq;q)_{n}}& = \frac{-a^{-1}(B;q)_{\infty}(-Abq;q)_{\infty}}{(-bq;q)_{\infty}(-aq;q)_{\infty}} \sum_{m=0}^{\infty} \frac{(A^{-1};q)_{m} \left(Abq/a\right)^{m}}{\left(-B/a;q\right)_{m+1}}\\
&\qquad+ (1+b) \sum_{m=0}^{\infty} \frac{(-a^{-1};q)_{m+1} \left(-ABq/a;q\right)_{m} (-b)^{m}}{(-B/a;q)_{m+1} (Abq/a;q)_{m+1}}.
\end{align*}
Replacing  $q$ by $q^2$ and letting $a=-q$, $b=-q^2$, $B=q^2$, $A=1/q^2$ in the above identity and then substituting the resultant in \eqref{before rlnb}, we are led to
\begin{align}\label{before 2phi1}
\mathcal{S}_{\mathcal{OE}}(q)&=\frac{q^2}{(1-q^2)(q^2;q^2)_\infty(1-q)}\left[\frac{(q^2;q^2)_\infty(q^2;q^2)_\infty}{q(q^4;q^2)_\infty(q^3;q^2)_\infty}\sum_{m=0}^{\infty}\frac{(q^2;q^2)_m}{(q;q^2)_{m+1}}q^m\right. \nonumber\\
&	\qquad\left. +(1-q^2)\sum_{m=0}^{\infty}\frac{(\frac{1}{q};q^2)_{m+1}(q;q^2)_m}{(q;q^2)_{m+1}(q;q^2)_{m+1}}q^{2m} \right] \nonumber\\
&=	\frac{q}{(1-q)(q;q^2)_\infty}\sum_{m=0}^{\infty}\frac{(q^2;q^2)_m}{(q^3;q^2)_{m}}q^m-\frac{1}{(q^2;q^2)_\infty}\sum_{m=0}^{\infty}\frac{q^{2m+1}}{(1-q^{2m+1})^2}.
\end{align}
Invoking \eqref{2phi1} with replacing $q$ by $q^2$ and letting  $a=q^2$, $b=q^2$, $c=q^3$ and $z=q$, we get
\begin{align*}
\sum_{m=0}^{\infty}\frac{(q^2;q^2)_m}{(q^3;q^2)_{m}}q^m=\frac{1}{1-q}\sum_{m=0}^{\infty}\frac{(q;q^2)_m}{(q^3;q^2)_m(q^3;q^2)_{m}}(-1)^mq^{3m}q^{m(m-1)}.
\end{align*}
We substitute the above expression in \eqref{before 2phi1} so as to obtain
\begin{align*}
	\mathcal{S}_{\mathcal{OE}}(q)
&=\frac{1}{(q;q^2)_\infty}\sum_{m=0}^{\infty}\frac{(-1)^mq^{m^2+2m+1}}{(1-q^{2m+1})(q;q^2)_{m+1}}-\frac{1}{(q^2;q^2)_\infty}\left(\sum_{m=1}^{\infty}\frac{q^{m}}{(1-q^{m})^2}-\sum_{m=1}^{\infty}\frac{q^{2m}}{(1-q^{2m})^2}\right).
\end{align*}
Replacing $m$ by $m-1$ in first sum on the right-hand side and invoking the definitions of $\mathcal{U}(q)$ and $E_2(\tau)$, we complete the proof of the theorem.
\end{proof}

\begin{proof}[\textbf{Corollary \textup{\ref{relation bw sptEO and sptOE}}}][]
This is immediate from Theorems \ref{sptEO thm} and \ref{sptOE thm}.
\end{proof}

\begin{proof}[\textbf{Theorem \textup{\ref{sptEOud thm}}}][]
The definition of $\mathcal{S}_{\mathcal{EO}_{ud}}(q)$, given in \eqref{sptEOud}, implies that
\begin{align*}
	\mathcal{S}_{\mathcal{EO}_{ud}}(-q)&=\sum_{n=1}^{\infty}\frac{q^{2n}(q^{2n+1};q^2)_\infty}{(q^2;q^2)_n(1-q^{2n})}\\
	&=(q;q^2)_\infty\sum_{n=1}^{\infty}\frac{q^{2n}}{(q^2;q^2)_n(q;q^2)_n(1-q^{2n})}\\
	&=(q;q^2)_\infty\sum_{n=1}^{\infty}\frac{q^{2n}}{(q;q)_{2n}(1-q^{2n})}.
\end{align*}
Expanding the term $1/(1-q^{2n})$ through geometric series and then interchanging the order of summations, we are led to
\begin{align}\label{before euler}
	\mathcal{S}_{\mathcal{EO}_{ud}}(-q)
	&=(q;q^2)_\infty\sum_{m=0}^{\infty}\sum_{n=1}^{\infty}\frac{q^{2n(m+1)}}{(q;q)_{2n}} \nonumber\\
	&=\frac{(q;q^2)_\infty}{2}\sum_{m=0}^{\infty}\sum_{n=1}^{\infty}\frac{q^{n(m+1)}}{(q;q)_{n}}(1+(-1)^n)  \nonumber\\
	&=\frac{(q;q^2)_\infty}{2}\sum_{m=0}^{\infty}\sum_{n=1}^{\infty}\frac{q^{n(m+1)}}{(q;q)_{n}}+\frac{(q;q^2)_\infty}{2}\sum_{m=0}^{\infty}\sum_{n=1}^{\infty}\frac{(-1)^nq^{n(m+1)}}{(q;q)_{n}}.
\end{align}
Now using the  Euler's identity \cite[(2.2.5)]{andrews}
$\sum_{n=0}^{\infty}\frac{t^n}{(q;q)_n}=\frac{1}{(t)_\infty},$
twice in \eqref{before euler}, we see that
\begin{align}\label{before sum of tails}
	\mathcal{S}_{\mathcal{EO}_{ud}}(-q)&=\frac{(q;q^2)_\infty}{2}\sum_{m=0}^{\infty}\left(\frac{1}{(q^{m+1};q)_\infty}-1\right)+\frac{(q;q^2)_\infty}{2}\sum_{m=0}^{\infty}\left(\frac{1}{(-q^{m+1};q)_\infty}-1\right)\nonumber\\
	&=-\frac{(q;q^2)_\infty}{2(q;q)_\infty}\sum_{m=0}^{\infty}\big((q;q)_\infty-(q;q)_m\big)-\frac{(q;q^2)_\infty}{2(-q;q)_\infty}\sum_{m=0}^{\infty}\big((-q;q)_\infty-(-q;q)_m\big).
\end{align}
Next, we need the famous sum-of-tails identity of Ramanujan \cite[p.~14]{ram}
\begin{align}\label{ramanujan sum of tails}
\sum_{n=0}^{\infty}\big((-q;q)_\infty-(-q;q)_n\big)
&=-(-q;q)_\infty D(q)+\frac{1}{2}\sigma(q),
\end{align}
and the following identity of Zagier \cite[p.~951, Theorem 2]{Zagier}
\begin{align}\label{zagier}
\sum_{n=0}^{\infty}\left((q;q)_\infty-(q;q)_n\right)
=\frac{1}{2}H(q)-(q;q)_\infty D(q).
\end{align}
Equations \eqref{before sum of tails}, \eqref{ramanujan sum of tails} and \eqref{zagier} together yield
\begin{align*}
(-q;q)_\infty 	\mathcal{S}_{\mathcal{EO}_{ud}}(-q)&=-\frac{1}{2(q;q)_\infty}\left(\frac{1}{2}H(q)-(q;q)_\infty D(q)\right)+\frac{1}{2}D(q)-\frac{1}{4(-q;q)_\infty}\sigma(q)\\
&=D(q)-\frac{1}{4(q;q)_\infty}H(q)-\frac{1}{4(-q;q)_\infty}\sigma(q).
\end{align*}
This completes the proof of the theorem.
\end{proof}


\section{Concluding remarks and Further Questions}\label{section: concluding remarks}
In this paper, among other things, we show that the generating functions for partitions separated according to the parity of their parts and of their smallest part analgoues can be expressed in terms of mock modular forms and mock Maass theta functions. We note that the mock modular form $\mathcal{U}(q)$ has several representations, arising both from our work and from the existing literature. For the convenience, we collect below most of these representations in one place:
\begin{align}\label{uq several repres}
\mathcal{U}(q)&=q(q^4;q^4)_\infty\overline{\mathrm{oe}}(q)=(q^2;q^2)_\infty\mathcal{S}_{\mathcal{EO}}(q)=\sum_{n=1}^{\infty}H(8n-1)q^n=\sum_{n=1}^{\infty}\frac{(-1)^{n+1}q^{n^2}}{(1-q^{2n-1})(q;q^2)_{n}}\nonumber\\
&=\frac{(q; q^2)_{\infty}}{(q^2; q^2)_{\infty}} \sum_{n \in \mathbb{Z}} \frac{ q^{2n^2+3n+1}}{\left(1 - q^{2n+1}\right)^2}=\sum_{n=0}^{\infty}\frac{(q^2;q^2)_n}{(q;q^2)_{n+1}}q^{n+1}=\sum_{n=1}^\infty\sum_{m=1}^n\frac{q^{\frac{n(n+1)}{2}-\frac{m(m-1)}{2}}}{1-q^n},
\end{align}
where the last expression can be found in \cite[Equation (4.3)]{cc}. As discussed above in this paper, each of these representations of $\mathcal{U}(q)$ is useful from one perspective or another. For example, Alfes, Bringmann, and Lovejoy \cite{abl} used the fifth expression to show that it is a mock modular form, whereas the third, fourth, sixth, and seventh expressions in \eqref{uq several repres} yield Corollary \ref{U combi} (see Theorems \ref{U in terms of difference of odd Ferrers} and \ref{U in terms of difference of simple}).

Now, it is natural to seek a bijection between the sets containing the partitions in \eqref{1t}. Moreover, our study in this paper naturally leads to the following questions:
\begin{enumerate}
	
	\item It will be interesting to find bijectives proofs for the results in  Corollary \ref{U combi}, Corollary \ref{oesoe}, Corollary \eqref{sigma* results},  Corollary \ref{partition difference1}, and Corollary \ref{partition difference2}.

	\item Bringmann, Craig, Nazaroglu \cite{bcn2} gave an exact formula for the partition function $p_{od}^{eu}(n)$, which counts the number of partitions of n whose even parts are larger than odd parts, and where
	odd parts must be distinct and likewise even parts are unrestricted. Its generating function is related to false indefinite theta functions and Maass forms. Note that our function $\mathcal{K}(q)$ in Theorem \ref{sigma* as generating function} exhibits similar behavior. While we have provided only the main asymptotic term of $k(n)$, it would be interesting to obtain an exact formula for $k(n)$.
	\item Similarly, one may ask whether an exact formula for $\mathcal{L}_e(n)$ can be obtained, for which the main asymptotic term is given in Corollary \ref{Len asymptotic}.
	\item Asymptotic expansions for the partition functions $\overline{\mathcal{OE}}(n)$, $\mathrm{spt}_{\mathcal{EO}}(n)$, $\mathrm{spt}_{\mathcal{OE}}(n)$ and $\mathrm{spt}_{\mathcal{EO}_{ud}}(n)$?
	\item In this paper, we considered smallest-part analogues only for three of the partition functions separated according to the parity of their parts. It would be of worth to investigate whether the smallest-part analogues of other partition functions in this class also have interesting properties.
	\item Do the partition functions considered in this paper satisfy any interesting congruence properties?

\end{enumerate}

\noindent {\bf{Acknowledgements:}}\,
Authors would like to thank Koustav Banerjee for his valuable comments on an earlier version of the paper. They also thank Professor George Andrews and Professor Atul Dixit for helpful discussions. The second author is partially supported by Grant ANRF/ECRG/ 2024/003222/PMS of the Anusandhan National Research Foundation (ANRF), Govt. of India, and  the FIG grants of IIT Roorkee. The first author is supported by UGC fellowship for her PhD. Both the authors sincerely thank these institutions for the support.

\noindent 
\subsection*{AI Declaration}\, This work was done without the use of any artificial intelligence tools. It is solely the work of the authors.


\begin{thebibliography}{00}


\bibitem{aa}
S.~Ahlgren and N.~Andersen, \emph{Algebraic and transcendental formulas for the smallest parts function}, Adv. Math. \textbf{289} (2016) 411--437.

\bibitem{abl}
C.~Alfes, K.~Bringmann and J.~Lovejoy, \emph{Automorphic properties of generating functions for generalized odd rank moments and odd Durfee symbols}, Math. Proc. Camb. Phil. Soc. (2011), \textbf{151}, 385--406.

\bibitem{andrews rlnb}
G.~E.~Andrews, \emph{Ramanujan’s ``lost” notebook. I. Partial $\theta$-functions}, Adv. in Math. \textbf{41} (1981), no. 2, 137--172.


\bibitem{andrews}
G.~E.~Andrews, \emph{The Theory of Partitions}, Cambridge University Press (1998).

\bibitem{andrews invent}
G.~E.~Andrews, \emph{Durfee symbols, and the Atkin Garvan moments of ranks}, Invent. Math. \textbf{169} (1), 37--73 (2007).

\bibitem{andrews spt}
G.~E.~Andrews, \emph{The number of smallest parts in the partitions of $n$}, J. Reine Angew. Math. \textbf{624} (2008), 133--142.


\bibitem{andrews parity1}
G.~E.~Andrews, \emph{Integer partitions with even parts below odd parts and the Mock theta functions}, Ann. Comb. \textbf{22} (2018) 433--445.

\bibitem{andrews parity2}
G.~E.~Andrews, \emph{Partitions with parts separated by parity}, Ann. Comb. \textbf{23} (2019) 241--248.

\bibitem{rlnb-V}
G.~E.~Andrews and B.~C.~Berndt, \emph{Ramanujan's Lost Notebook, Part V}, Springer, New York (2018).

\bibitem{adh}
G.~E.~Andrews, F. ~J.~Dyson and D.~Hickerson, \emph{Partitions and indefinite quadratic forms}, Invent. Math. \textbf{91} (1988), 391--407.

\bibitem{agl}
G.~E.~Andrews, F.~G~Garvan and J.~Liang, \emph{Self-conjugate vector partitions and the parity of the spt-function}, Acta Math. \textbf{158} (3) (2013).


\bibitem{an}
G. ~E.~Andrews and D.~Newman, \emph{The minimal excludant in integer partitions}, J. Integer Seq. \textbf{23} (2020).

\bibitem{bailey}
W.~N.~Bailey, \emph{Generalized hypergeometric series}, (Reprinted: Hafner, New York, (1964), Cambridge Mathematical Tract No. 32, Cambridge University Press, London, 1935.

\bibitem{bbd}
K.~Banerjee, K.~Bringmann and A.~Dixit, \emph{Restricted Overpartitions and concave compositions: their modularity and asymptotics}, preprint, 2026.

\bibitem{bringmann duke}
 K.~Bringmann, \emph{On the explicit construction of higher deformations of partition statistics}, Duke Math. J., \textbf{144} (2):195--233, 2008.

\bibitem{bfor}
K.~Bringmann, A.~Folsom, K.~Ono and L.~Rolen, \emph{Harmonic Maass forms and mock modular forms: theory and applications}, AMS colloquium series \textbf{64}, 2017.

\bibitem{bcn}
K.~Bringmann, W.~Craig and C.~Nazaroglu, \emph{On the asymptotic behavior for partitions separated by parity}, {Electron. J. Combin.} \textbf{32:1} (2025), 1--18.

\bibitem{bcn2}
K.~Bringmann, W.~Craig and C.~Nazaroglu, \emph{Precision asymptotic for partitions featuring false-indefinite theta functions}, preprint, 2025.

\bibitem{bjs}
K.~Bringmann and C.~Jennings-Shaffer, \emph{A note on Andrews’ partitions with parts separated by parity}, Ann. Comb. \textbf{23}, 573--578 (2019).

\bibitem{bjm}
K.~Bringmann, C.~Jennings-Shaffer and K.~Mahlburg, \emph{On a Tauberian theorem of Ingham and Euler--Maclaurin	summation}, Ramanujan J. \textbf{61} (2021), 55--86.

\bibitem{bl}
K.~Bringmann and J.~Lovejoy, \emph{Overpartitions and class numbers of binary quadratic forms}, Proc. Natl. Acad. Sci. USA \textbf{106} (2009), no. 14, 5513--5516.

\bibitem{blo}
K.~Bringmann, J.~Lovejoy and R.~Osburn, \emph{Automorphic Properties of Generating Functions for Generalized Rank Moments and Durfee Symbols}, Int. Math. Res. Not. IMRN (2010), 238--260.

\bibitem{bopr}
J.~Bryson, K.~Ono, S.~Pitman and R.~C.~Rhoades, \emph{Unimodal sequences and quantum and mock modular forms}, Proc. Natl. Acad. Sci. USA \textbf{109} (2012), no. 40, 16063--16067.

\bibitem{cc}
D.~Chen and R.~Chen, \emph{Generating functions of the Hurwitz class numbers associated with certain mock theta functions}, Ramanujan J. \textbf{(64)} (2024), 371--400.

\bibitem{cg}
R.~Chen and F.~G.~Garvan, \emph{A proof of the mod $4$ unimodal sequence conjectures and related mock theta functions}, Adv. Math. \textbf{398} (2022) 108235.

\bibitem{cohen}
H.~Cohen, \emph{$q$-identities for Maass waveforms}, Invent. Math. \textbf{91} (1988), 409--422.

\bibitem{dy}
A.~Dixit and A.~J.~Yee, \emph{Generalized higher order spt-functions}, Ramanujan J. \textbf{31} (1-2) (2013) 191--212.

\bibitem{fine}
N.~J.~Fine, \emph{Basic hypergeometric series and applications}, Mathematical Surveys and Monographs, vol. 27, American Mathematical Society, Providence, RI, 1988.

\bibitem{fk}
A.~Folsom, C.~Ki, Y.~N.~T.~Vu and B.~Yang, \emph{``Strange” combinatorial quantum modular forms}, J. Number Theory \textbf{170} (2017), 315--346.

\bibitem{fo}
A.~Folsom and K.~Ono, \emph{The spt-function of Andrews}, Proc. Natl. Acad. Sci. USA \textbf{105} (51) (2008) 20152--20156.

\bibitem{Fu-Tang}
S.~Fu and D.~Tang, \emph{Partitions with parts separated by parity: conjugation, congruences and the mock theta functions},  Proc. Roy. Soc.
Edinburg Sect. A: Mathematics \textbf{155} (2025), 954–974.

\bibitem{gr}
G.~Gasper and M.~Rahman, \emph{Basic hypergeometric series}, 2nd ed., Encyclopedia of Mathematics and its Applications, vol. 96, Cambridge University Press, Cambridge, 2004.

\bibitem{hr}
G.~Hardy and S.~Ramanujan, \emph{Asymptotic formulae in combinatory analysis}, Proc. Lond. Math. Soc. (3) \textbf{17} (1918), 75--115.

\bibitem{hj}
K.~Hikami and J.~Lovejoy, \emph{Torus knots and quantum modular forms}, Res. Math. Sci. \textbf{2} (2015), 15 pp.

\bibitem{iw}
K.~Iwasawa, \emph{Lectures on $p$-Adic $L$-Functions}, Princeton University Press and University of Tokyo Press, 1972.

\bibitem{lovejoy}
J.~Lovejoy, \emph{Quantum $q$-series identities}, Hardy-Ramanujan J. (Special Commemorative Volume in Honour of Srinivasa Ramanujan) 44, 61--73 (2021).

\bibitem{mc}
R.~J.~McIntosh, \emph{Second order mock theta functions}, Can. Math. Bull. \textbf{50} (2) (2007) 284--290.

\bibitem{o}
K.~Ono, \emph{Congruences for the Andrews spt function}, Proc. Natl. Acad. Sci. USA, \textbf{108}(2):473--476, 2011.

\bibitem{ram}
S.~Ramanujan, \emph{The Lost Notebook and Other Unpublished Papers}, Narosa, New Delhi (1988).

\bibitem{titch}
E.~C.~Titchmarsh, \emph{The Theory of the Riemann Zeta Function}, Clarendon Press, Oxford, 1986.

\bibitem{wash}
 L.~C.~Washington, \emph{Introduction to cyclotomic fields}, second edition, GTM 83, Springer-Verlag, New York, 1997.

\bibitem{watson}
G.~N.~Watson, \emph{The final problem: an account of the mock theta functions}, J. London Math. Soc. \textbf{11} (1), 55--80 (1936).

\bibitem{Zagier}
D.~Zagier, \emph{Vassiliev invariants and a strange identity related to
the Dedekind eta-function}, Topology \textbf{40} (2001) 945--960.

\bibitem{Zagier quantum}
D.~Zagier, \emph{Quantum Modular Forms}, Quanta of maths, 659--675, Clay Math. Proc., 11, Amer. Math. Soc., Providence, RI, 2010.

\bibitem{Zwegers}
S.~Zwegers, \emph{Mock Maass theta functions}, Q. J. Math. \textbf{63} (2012), 753--770.

%

%


%
%
%
%
%
%





\end{thebibliography}
\end{document}